\documentclass[12pt]{amsart}  
\usepackage{amsmath,amssymb,latexsym, amsthm, amscd, mathrsfs, stmaryrd,enumerate}
\usepackage[linktocpage=true]{hyperref}
\usepackage{color}
\usepackage[parfill]{parskip}
\usepackage[all]{xy}
 
\newtheorem{thm}{Theorem} [section]
\theoremstyle{definition}

\newcommand{\Top}{{\mathrm{Top}}}
\newcommand{\Soc}{{\mathrm{Soc}}}

\newcommand{\tto}{\twoheadrightarrow}
\newtheorem{Question}[thm]{Question}
\newtheorem{example}[thm]{Example}
\newtheorem{rem}[thm]{Remark}
\theoremstyle{plain}
\newtheorem{prop}[thm]{Proposition}
\newtheorem{lem}[thm]{Lemma}
\newtheorem{cor}[thm]{Corollary}

\numberwithin{equation}{section}

\newcommand{\add}{\mathrm{add}}
\newcommand{\Spec}{\mathrm{Spec}}
\newcommand{\Prim}{\mathrm{Prim}}
\newcommand{\bL}{\mathbf{L}}
\newcommand{\Hom}{\mathrm{Hom}}

\newcommand{\ad}{{\text{ad}}}

\newcommand{\C}{\mathbb C}

\newcommand{\h}{\mathfrak{h}}

\newcommand{\la}{\lambda}
\newcommand{\mc}{\mathcal}
\newcommand{\mf}{\mathfrak}

\newcommand{\N}{\mathbb N}
\newcommand{\one}{{\ov 1}}

\newcommand{\ov}{\overline}

\newcommand{\Z}{\mathbb Z}
\newcommand{\oo}{{\ov 0}}

\newcommand{\vare}{\varepsilon}

\newcommand{\Ula}{\ov{U}_\la}

\newcommand{\cP}{\mathcal{P}}
\newcommand{\cF}{\mathcal{F}}
\newcommand{\fk}{\mathfrak{k}}
\newcommand{\fg}{\mathfrak{g}}
\newcommand{\fh}{\mathfrak{h}}
\newcommand{\fn}{\mathfrak{n}}
\newcommand{\fb}{\mathfrak{b}}
\newcommand{\oa}{{\bar{0}}}
\newcommand{\ob}{{\bar{1}}}
\newcommand{\Ind}{{\rm Ind}}
\newcommand{\Coind}{{\rm Coind}}
\newcommand{\Res}{{\rm Res}}
\newcommand{\coker}{{\rm coker}}
\newcommand{\cH}{\mathcal{H}}
\newcommand{\cL}{\mathcal{L}}
\newcommand{\cO}{\mathcal{O}}
\newcommand{\Ann}{{\rm{Ann}}}
\newcommand{\LAnn}{{\rm{LAnn}}}
\newcommand{\op}{{\text{op}}}
\newcommand{\mC}{\mathbb{C}}
\newcommand{\mZ}{\mathbb{Z}}
\newcommand{\cB}{\mathcal{B}}
\newcommand{\id}{\mathrm{id}}

\title[Category~$\cO$ for the periplectic superalgebra]{The primitive spectrum and category~$\cO$ for the periplectic Lie superalgebra }

\author[Chih-Whi Chen]{Chih-Whi Chen}
\address{Department of Mathematics, Uppsala University, Box. 480,
	SE-75106, Uppsala, SWEDEN} \email{chih-whi.chen@math.uu.se}

\author[Kevin Coulembier]{Kevin Coulembier}
\address{School of Mathematics and Statistics, University of Sydney, AUSTRALIA} \email{kevin.coulembier@sydney.edu.au}

\date{}

\keywords{Harish-Chandra bimodules, periplectic Lie superalgebra, twisting functors, completion functors, primitive spectrum, category O, block decomposition}
\subjclass[2010]{
	16E30, 17B10}

\begin{document}

\begin{abstract}
We solve two problems in representation theory for the periplectic Lie superalgebra $\mathfrak{pe}(n)$, namely the description of the primitive spectrum in terms of functorial realisations of the braid group and the decomposition of category~$\cO$ into indecomposable blocks.

To solve the first problem we establish a new type of equivalence between category~$\cO$ for all (not just simple or basic) classical Lie superalgebras and a category of Harish-Chandra bimodules. The latter bimodules have a left action of the Lie superalgebra but a right action of the underlying Lie algebra. To solve the second problem we establish a BGG reciprocity result for the periplectic Lie superalgebra.
  \end{abstract}
\maketitle

\setcounter{tocdepth}{1}


 %
 %
   \vskip 1cm
  \section{Introduction}
  We study some aspects of the representation theory of the periplectic Lie superalgebra $\mathfrak{pe}(n)$. These algebras constitute one of the four families of algebras which appear, along with some exceptional ones, in the classification of simple classical Lie superalgebras, see~\cite{ChWa12, Mu12}. Note that $\mathfrak{pe}(n)$ is not actually simple itself, but has a simple subalgebra $\mathfrak{spe}(n)$ of codimension $1$. Unlike the Lie superalgebras $\mathfrak{gl}(m|n)$ and $\mathfrak{osp}(m|2n)$, the periplectic Lie superalgebra is not `basic', meaning it does not have a non-degenerate invariant bilinear form.
  
The primitive spectrum of a Lie (super)algebra $\fg$ refers to the set of annihilator ideals in~$U(\fg)$ of the simple modules, partially ordered with respect to inclusions.  The description of the primitive spectrum of a reductive Lie algebra is a classical result, with the last piece of the proof obtained in~\cite{Vogan}, see \cite[Section~15.3]{Mu12} for an overview. 
 In \cite{Co16} it was proved that the primitive spectrum of a {\em simple basic} classical Lie superalgebra can be described in terms of the combinatorics of the twisting functors on category~$\cO$. For the case $\mathfrak{gl}(m|n)$ this even led to a description of the primitive spectrum in terms of the ${\rm Ext}^1$-quiver of category~$\cO$. An essential ingredient in the construction in~\cite{Co16} was the equivalence between $\cO$ and a category of Harish-Chandra bimodules, see~\cite{BeGe, MaMe12}.

The first complication in representation theory of~$\mathfrak{pe}(n)$ lies in the existence of a Jacobson radical in the universal enveloping algebra, see~\cite{Se02}, which prevents the existence of `typical modules' with same properties as for other simple classical Lie superalgebras.
One of the consequences is that for $\mathfrak{pe}(n)$, contrary to the other simple classical Lie superalgebras, there was no known theory of Harish-Chandra bimodules available for~$\mathfrak{pe}(n)$.  We resolve this by studying an alternative version of Harish-Chandra bimodules, in terms of bimodules with a left action of the Lie superalgebra and a right action of its underlying Lie algebra. We prove an equivalence between category~$\cO$ and a category of such Harish-Chandra bimodules, for all (not just simple or basic) classical Lie superalgebras. Using this result we can then derive a descriptions of the primitive spectra for {\em all} classical Lie superalgebras in terms of translation functors on Harish-Chandra bimodules, in the spirit of~\cite{Vogan}.

Subsequently, we introduce Enright's completion functors on category~$\cO$ for arbitrary classical Lie superalgebras and prove that their combinatorics governs the primitive spectrum. Completion functors yield an action of the braid group (of the Weyl group of the underlying Lie algebra). Another such action on category~$\cO$ is given by the twisting functors, see e.g.~\cite{AS, CoMa14}. For Lie superalgebras of type I, we prove that whenever $\cO$ admits a suitable duality, twisting and completion functors are isomorphic up to conjugation with this duality, as is well-known in various specific cases, see~\cite{KM, CoMa14}. In particular, this allows us to express the primitive spectrum for $\mathfrak{pe}(n)$ in terms of twisting functors, as an extension of the main result of~\cite{Co16}.

For completeness we also generalise the known equivalences between category~$\cO$ and Harish-Chandra bimodules, see~\cite{MaMe12}, to $\mathfrak{pe}(n)$. Using a slight generalisation of the conventional proof we can construct equivalences based on Verma modules which are not necessarily `typical', but still satisfy Kostant's problem. This applies to $\mathfrak{pe}(n)$ by \cite{Se02}.

The second complication in representation theory of~$\mathfrak{pe}(n)$ is that its universal enveloping algebra has a very small centre, see~\cite{Go2001}.
Consequently, the block decomposition of the category of finite dimensional modules and of category~$\cO$ is not controlled by the combination of the centre and the root lattice. For finite dimensional modules, the block decomposition was recently obtained independently in~\cite{B+9} and \cite{Copn}, with one direction already proved earlier in~\cite{Ch}.
To determine the blocks in~$\cO$ we establish a BGG reciprocity result and exploit the technique in~\cite{B+9} of decomposing the translation functors using a `fake Casimir operator'. 
The block decomposition for $\cO$ was also obtained in unpublished work by Inna Entova-Aizenbud and Vera Serganova.

The paper is organised as follows. In Section~\ref{SecPrel} we recall some general notions in Lie superalgebra theory. In Section~\ref{SecHC} we obtain our results on Harish-Chandra bimodules. In Section~\ref{SecPrim} we observe that Musson's result (and the methods to prove it) of~\cite{Mu92} extend from simple to arbitrary classical Lie superalgebras.  Subsequently, we study twisting and completion functors, and their relation with the primitive spectrum. In Section~\ref{SecO} we study the BGG category~$\cO$ for $\mathfrak{pe}(n)$. In Section~\ref{Sec2} we focus on the specific case $\mathfrak{pe}(2)$, for which we determine the characters of the simple highest weight modules, classify the blocks in~$\cO$ up to equivalence, show that generic blocks are Koszul and give an explicit description of the primitive spectrum.

\subsection*{Acknowledgement}  The first author is supported by Vergstiftelsen and the second author is supported by ARC grant DE170100623.
Both authors thank the Institute of Mathematics of Academia Sinica in Taiwan for hospitality and support and Shun-Jen Cheng, Ian Musson and Weiqiang Wang  for useful discussions. The first author has learned the validity of the BGG reciprocity (Proposition \ref{Thm::BGGrcp}) and the linkage principle of category~$\mc O$ for~$\mf{pe}(2)$ from Shun-Jen Cheng and Weiqiang Wang.
   
\section{Preliminaries}\label{SecPrel}
For the entire paper, we fix the field of complex numbers $\mC$ as ground field.

\subsection{Super Algebra}\label{Z2}
We will always work with~$\mZ_2$-graded algebras, vector spaces and modules, which we will regard as `super' algebras and vector spaces. Morphisms in the category of super vector spaces are assumed to preserve the $\mZ_2$-grading, and the same thus holds for morphisms of superalgebras or modules over superalgebras. Unless specified otherwise, we consider left modules. For a homogeneous element $X$ of a $\mZ_2$-graded vector space we denote its parity by~$\overline{X}\in\{\bar{0},\bar{1}\}=\mZ_2$.
For any Lie superalgebra $\fg=\fg_{\oa}\oplus\fg_{\ob}$, see~\cite{Mu12}, we denote its universal enveloping algebra by~$U=U(\fg)$ and also write $U_0=U(\fg_{\oa})$. On the category of super vector spaces we denote the parity reversal functor by~$\Pi$. 

Any anti-automorphism $\sigma:\fg\to \fg$ of Lie superalgebras is an isomorphism of~$\mZ_2$-graded vector spaces satisfying
$$\sigma([X,Y])=(-1)^{\overline{X}.\overline{Y}}[\sigma(Y),\sigma(X)],$$
for all homogeneous elements $X, Y \in \mf g$. 
Such an anti-automorphism extends to an anti-isomorphism of the associative superalgebra $U(\fg)$.
As an example, we have the algebra anti-involution $t:U\to U$, which maps $X\in \fg$ to~$-X$. For a left (resp. right) $U$-module $M$, we denote by~$M^t$ the corresponding right (resp. left) $U$-module, obtained via twisting by~$t$. We have the restriction functor $\Res:=\Res^{\fg}_{\fg_{\oa}}$ from $\fg$-modules to~$\fg_{\oa}$-modules and its left adjoint $\Ind:=\Ind^{\fg}_{\fg_{\oa}}=U\otimes_{U_0}-$ and right adjoint $\Coind^{\fg}_{\fg_{\oa}}=\Hom_{U_0}(U,-)$. We will usually leave out the references to~$\fg$ and $\fg_{\oa}$ in this notation. 

For a Lie superalgebra $\fg$, we call an element $X_0\in\fg_{\oa}$ a {\em grading element} if its adjoint action on $\fg$ is semisimple, with eigenvalues in~$\mZ$ such that the induced $\mZ$-grading satisfies
$$\fg_{\oa}\;=\;\bigoplus_{i\in2\mZ}\fg_i\quad\mbox{and}\quad \fg_{\ob}=\bigoplus_{i\in 2\mZ+1}\fg_i.$$
Examples of Lie superalgebras with such an element are all reductive Lie algebras and $\mathfrak{osp}(m|2n)$. More examples will be given in Section~\ref{typeI0}.

\subsection{Classical Lie superalgebras}
We call a Lie superalgebra $\fg$ over $\mC$ {\bf classical}  provided $\dim_{\mC}\fg$ is finite, the Lie algebra $\fg_{\oa}$ is reductive and $\fg_{\ob}$ is a semisimple $\fg_{\oa}$-module for the adjoint action. We do {\em not} assume that~$\fg$ is simple.


\subsubsection{Finite dimensional modules}

We denote by~$\cF^{\oa}$ the category of finite dimensional semisimple $\fg_{\oa}$-modules. We denote by~$\cF=\cF(\fg,\fg_{\oa})$ the category of finite dimensional $\fg$-modules which restrict to objects in~$\cF^{\oa}$. Since $\Ind$ maps from $\cF^{\oa}$ to~$\cF$, with right adjoint given by the exact restriction functor $\Res:\cF\to\cF^{\oa}$, the category~$\cF$ has enough projective modules. The projective modules are precisely the direct summands of modules $\Ind V$, for arbitrary~$V\in\cF^{\oa}$. We denote the full subcategory of projective modules in~$\cF$ by~$\cP$.
The natural duality on $\cF$ is given by~$M\mapsto (M^\ast)^t$, with $M^\ast=\Hom_{\mC}(M,\mC)$.

 For an arbitrary~$\fg$-module $M$, we will introduce some {\em full} subcategories of the category of all $U$-modules. We denote by~$\cP\otimes M$, resp. $\cF\otimes M$, the category of all $\fg$-modules of the form $V\otimes M$, with~$V\in \cP$, resp. $V\in \cF$. By $\add(\cP\otimes M)$ and $\add(\cF\otimes M)$ we denote the respective categories of all direct summands of modules in the former categories. 
 By $\langle \cP\otimes M\rangle=\langle \cF\otimes M\rangle$ we denote the abelian category of subquotients of modules in~$\cP\otimes M$ and by~$\coker (\cP\otimes M)$, resp. $\coker(\cF\otimes M)$, the category of modules which are presented by modules in~$\cP\otimes M$, resp. $\cF\otimes M$.

\subsubsection{Borel and Cartan subalgebras} \label{SectBC}
We choose a Cartan subalgebra $\fh_{\oa}$ of~$\fg_{\oa}$. The non-zero weights appearing in the adjoint representation of~$\fg$, with respect to~$\fh_{\oa}$, are then denoted by~$\Phi=\Phi_{\oa}\cup\Phi_{\ob}\subset\fh^\ast$. We set $\Gamma=\mZ\Phi$ and let $\Upsilon\subset\fh^\ast_{\oa}$ denote the set of integral weights, that is, weights appearing in modules in~$\cF$, or equivalently in~$\cF^{\oa}$. By construction, $\Gamma\subset\Upsilon\subset \fh_{\oa}^{\ast}$ is a chain of subgroups.

Choose a Borel subalgebra $\fb_{\oa}\supset\fh_{\oa}$ of~$\fg_{\oa}$. We have a corresponding decomposition $\Phi_{\oa}=\Phi_{\oa}^{+}\sqcup\Phi_{\oa}^-$ into positive and negative roots. Let $\rho_{\oa}$ denote half the sum of all elements in~$\Phi_{\oa}^+$. We denote the $\rho_{\oa}$-shifted action of the Weyl group $W=W(\fg_{\oa}:\fh_{\oa})$ on $\fh_{\oa}^\ast$ by~$w\cdot\lambda=w(\lambda+\rho_{\oa})-\rho_{\oa}$.

Unless mentioned otherwise, we express properties of weights with respect to the structure of~$\fg_{\oa}$. A weight $\lambda\in\fh_{\oa}^\ast$ is regular if the size of the dot $W$-orbit of~$\lambda$ is $|W|$. A weight $\lambda$ is dominant if there exists no $w\in W$ such that~$w\cdot\lambda-\lambda$ is a non-empty sum of elements in~$\Phi_{\oa}^+$.
The integral Weyl group for $\lambda$ is the Coxeter group
$$W^\lambda\;=\;\{w\in W\,|\, w\cdot\lambda-\lambda\in\Gamma\}.$$
Clearly this group only depends on $\lambda+\Upsilon$.

For a $\fg$-weight module $M=\bigoplus_{\mu\in\fh_{\oa}^\ast}M^\mu$ with finite dimensional weight spaces, we let $\text{ch}M$ denote the $\mf g_\oa$-character of~$M$, namely, 
\begin{equation}\label{chDs}\text{ch}M : = \sum_{\mu\in \h_{\oa}^*}\text{dim}M^\mu e^{\mu}.\end{equation}

\subsubsection{Category~$\cO$}\label{SecDefO}
We fix a Cartan subalgebra $\mf h_\oa$ and Borel subalgebra $\mf b_\oa\supset\fh_{\oa}$ of~$\fg_{\oa}$. We let $\mf{n}^{+}_\oo$, resp. $\mf n^-_\oo$, be the subalgebras of~$\mf{g}_\oo$ spanned by positive, resp. negative, root vectors.

We denote by~$\cO=\cO(\fg,\fb_{\oa})$ the BGG category of~$\fg$-modules which are finitely generated, semisimple as $\fh_{\oa}$-modules and locally finite as $\fb_{\oa}$-modules. This is thus the category of~$\fg$-modules which are mapped by~$\Res$ to modules in the BGG category~$\cO^{\oa}:=\cO(\fg_{\oa},\fb_{\oa})$ of~\cite{BGG76}. For $\Lambda\subset \fh^\ast$ any subset closed under the action of~$\Gamma=\mZ\Phi$, we denote by~$\cO_\Lambda$ the full subcategory of modules with only non-zero weight spaces for elements in~$\Lambda$. For instance, we have $\cF\subset\cO_{\Upsilon}$, and a decomposition
$\cO\;\cong\;\bigoplus_{\Lambda\in \fh^\ast/\Gamma}\cO_{\Lambda}.$

Assume the Cartan subalgebra $\fh_{\oa}$ contains a grading element $X_0$ as in Section~\ref{Z2}. We choose a representative for each equivalence class in~$\mC/\mZ$. This allows to define a map $\pi_0:\mC\to\mZ_2$ by setting $\pi_0(a+i)$ equal to~$i\mbox{(mod 2)}$ for each such representative $a$ and $i\in\mZ$. This leads to a map
$$\pi:\fh^\ast_{\oa}\to\mZ_2,\quad \lambda\mapsto\pi_0(\lambda(X_0)).$$
We then have a decomposition
$$\cO\;=\; \cO_{red}\oplus\Pi \cO_{red},$$
where $\cO_{red}$ is the subcategory of all modules where each space $M^\mu$ is homogeneous of parity~$\pi(\mu)$. The category~$\Pi \cO_{red}$ is equivalent to~$\cO_{red}$ and contains all modules where each space $M^\mu$ has parity~$\pi(\mu)+\ob$. Furthermore, the category~$\cO_{red}$ is equivalent to the category with same objects as in~$\cO$, but where we allow all, not just grading preserving, morphisms.

\subsubsection{Dualities on $\cO$} An anti-involution $\sigma$ of~$\fg$ is a {\em good involution} (with respect to a given triangular decomposition of~$\fg_{\oa}$) if
$\sigma(\fh_{\oa})=\fh_{\oa}$ and $\sigma(\fn^+_{\oa})=\fn^-_{\oa}$.  This induces an involution $\sigma^\ast$ on $\fh^\ast_{\oa}$, where $\sigma^\ast(\lambda)(H)=\lambda(\sigma(H))$, for all $\lambda\in\fh^\ast_{\oa}$ and $H\in\fh_{\oa}$.

Define
$M^{\circledast}\;=\;\bigoplus_{\mu} (M^\mu)^\ast,$
which is a right $\fg$-module as a submodule of~$M^\ast=\Hom_{\mC}(M,\mC)$. Then define the left $U$-module $D_\sigma M$ with underlying vector space $M^{\circledast}$ and action given by
$$(X\alpha)(v)\;=\;(-1)^{\overline{X}.\overline{\alpha}}\alpha(\sigma(X)v),\quad\mbox{for~$\alpha\in M^{\circledast}$ and $v\in M$}.$$
This yields a contravariant auto-equivalence $D_\sigma$ of~$\cO$. We have
\begin{align} \label{dualchar} &\text{ch} D_\sigma M : = \sum_{\mu\in \h^*_\oo}\text{dim}M^\mu e^{\sigma^\ast\mu}.\end{align}
Furthermore, for any $\alpha\in\Phi$, we have
$\sigma(\fg_{\alpha})=\fg_{-\sigma^\ast(\alpha)}.$
In particular, $\sigma^\ast$ restricts to a bijection of the set of simple roots in~$\Phi_{\oa}^+$.

\begin{example}${}$\label{ExInv}
  \begin{enumerate}[(a)]
 	\item If $\mf g$ is one of the contragredient Lie superalgebras in~\cite[Theorem 5.1.5]{Mu12} or $\mathfrak{q}(n)$, then the antiautomorphism of~\cite[Proposition 8.1.6]{Mu12} satisfies the above properties. In this case, $D_\sigma$ is the contragredient duality of~\cite[Section 13.7]{Mu12} and $\sigma^\ast$ is the identity on $\fh_{\oa}^\ast$. 	
	
	\item For a reductive Lie algebra $\fg=\fg_{\oa}$, the anti-involution of~$\fg$ in~\cite[Section~0.5]{Hu08} is a special case of (a). In this case we get the simple preserving duality~$D_\sigma=(\cdot)^\vee$ on $\cO^{\oa}$ of~\cite[Section~3.2]{Hu08}.

	\item If $\mf g =\mf{pe}(n)$, will introduce an appropriate anti-involution in Section~\ref{SecTwi1}. 
 \end{enumerate}
 \end{example}

\subsubsection{Projective modules in~$\cO$} 
Fix $\Lambda\in\fh^\ast_{\oa}/\Upsilon$. For now we consider the Verma module $M^{\oa}_\lambda=U_0\otimes_{U(\fb_{\oa})}\mC_\lambda$ without specifying in which parity it is assumed to be, for~$\lambda\in\Lambda$.
\begin{lem} \label{pReDomi}
 Let $M\in\cO_\Lambda$ be such that~$\Res M$ is projective in~$\cO^{\oa}$ and contains as a direct summand some $M_\lambda^{\oa}$ for~$\lambda\in\Lambda$ regular and dominant. Then $\add(\cP\otimes M)$ is the category of projective modules in~$\cO_\Lambda$. If $M$ is projective in~$\cO$, then $\add(\cP\otimes M)$ is equal to~$\add(\cF\otimes M)$.
\end{lem}
\begin{proof}
Take an arbitrary~$V\in \cF^{\oa}$, then we have isomorphisms of functors 
\begin{eqnarray*}
\Hom_{\fg}(\Ind(V)\otimes M,-)&\cong& \Hom_{\fg}(\Ind(V),\Hom_{\mC}(M,-)^{\ad})\\
&\cong& \Hom_{\fg_{\oa}}(V\otimes \Res(M),\Res-).
\end{eqnarray*}
This functor is therefore exact on $\cO_\Lambda$.
Furthermore, it follows from \cite[Theorem~3.3]{BeGe} that~$\add(\cF^{\oa}\otimes M^{\oa}_\lambda)$ is the category of projective modules in~$\cO^{\oa}_\Lambda$, which shows that~$\add(\cP\otimes M)$ is the category of projective modules in~$\cO_\Lambda$. 

If $M$ is projective itself, then clearly all modules in the category~$\add(\cF\otimes M)$ are projective. The latter contains the category  $\add(\cP\otimes M)$ of projective modules, from which the conclusion follows.
\end{proof}

We have the following corollary. For the definition of the notions in part (ii), we refer to Section~\ref{SecI}
\begin{cor} \label{ReDomi}
Take $\lambda\in\Lambda$ regular and dominant.
\begin{enumerate}[(i)]
\item The category of projective modules in~$\cO_\Lambda$ is $\add(\cF\otimes \Ind M_{\la}^{\oa})$.
\item If $\fg$ is of type I, the category of projective modules in~$\cO_\Lambda$ is $\add(\cP\otimes M_{\la})$.
\end{enumerate}
	\end{cor}
	\begin{proof}
	By construction, $\Ind M_{\la}^{\oa}$ is projective in~$\cO$. Since $\mC$ is a direct summand of the $\fg_{\oa}$-module $\Lambda\fg_{\ob}$, we find that~$M_\lambda^{\oa}$ is a direct summand of~$\Res\Ind M_\lambda^{\oa}$. Hence, part (i) follows from Lemma~\ref{pReDomi}.
	
	For part (b), we have that~$\Res M_\lambda =\Lambda\fg_{-1}\otimes M^{\oa}_\lambda$ is projective and contains $M_\lambda^{\oa}$ as a direct summand. The conclusion thus also follows from Lemma~\ref{pReDomi}.
	\end{proof}


\subsection{Classical Lie superalgebras of type I}\label{SecI}
A classical Lie superalgebra $\mf g$ is said to be of type I if $\mf g$ has a $\Z_2$-compatible $\Z$-gradation  
$\mf g=\mf g_{-1}\oplus \mf g_0 \oplus \mf g_{1}$. Examples of such algebras are 
$$\mathfrak{gl}(m|n),\;\,\mathfrak{sl}(m|n),\;\, \mathfrak{psl}(n|n),\;\,\mathfrak{osp}(2|2n),\;\mf{pe}(n),\;\mf{spe}(n)=[\mf{pe}(n),\mf{pe}(n)],$$ 
see \cite{ChWa12} and \cite{Mu12} for a complete treatment of these Lie superalgebras.

\subsubsection{Type I-0}\label{typeI0}
We will say that a Lie superalgebra $\fg$ of type I is of type I-0 if the $\mZ$-grading is induced by a grading element $H_0\in\fg_{\oa}$ as in Section~\ref{Z2}. It is then automatic that~$H_0\in\fh_{\oa}$ and $\fh_{\oa}$ is its own commutator in~$\fg$, so we simply write $\fh=\fh_{\oa}$. In the above list, the Lie superalgebras $\mathfrak{gl}(m|n)$, $\mathfrak{osp}(2|2n)$ and $\mathfrak{pe}(n)$ are always of type~I-0. The superalgebra $\mathfrak{sl}(m|n)$ is of type I-0 if and only if $m\not=n$.

Throughout this paper, we set $\mf g_{\geq 0}:= \mf g_0 \oplus \mf g_1$ and $\mf g_{\leq 0}:= \mf g_{-1}\oplus \mf g_0$ for classical Lie superalgebras of type I-0.  We have the corresponding Borel subalgebra $\fb=\fb_{\oa}\oplus\fg_{1}$ of~$\fg$. We let $\Phi_{\ob}^+$, resp. $\Phi_{\ob}^-$, denote the set of weights appearing in~$\fg_{1}$, resp. $\fg_{-1}$. By assumption, $\Phi^+_{\ob}\cap\Phi^-_{\ob}=\emptyset$, and neither set contains the weight $0$. We also set $\Phi^+=\Phi_{\oa}^+\cup \Phi_{\one}^+$ and $\Phi^-=\Phi_{\oa}^-\cup \Phi_{\one}^-$, where we note that the unions are disjoint.
We also consider the partial order $\le$ on $\fh^\ast$ where $\mu\le \lambda$ if and only if $\lambda-\mu$ is a sum of elements in 
$$\Phi^+\cup (-\Phi^-)\;=\; \Phi^+_{\oa}\,\sqcup\, (\Phi_{\ob}^+\cup (-\Phi_{\ob}^-)).$$
We also denote by~$\rho_1$, resp. $\rho_{-1}$, half the sum of elements in~$\Phi_{\ob}^+$, resp. $\Phi_{\ob}^-$. Finally, we set $\omega=\rho_1+\rho_{-1}$. Note that we have $\omega=0$ for~$\mathfrak{gl}(m|n)$, $\mathfrak{osp}(2|2n)$ and $\mathfrak{spe}(n)$, but $\omega\not=0$ for~$\mathfrak{pe}(n)$.

We choose a function $\pi$ as in Subsection \ref{SecDefO} and henceforth denote $\cO_{red}$ simply by~$\cO$.

\subsubsection{Verma modules and Kac modules}  Fix a classical Lie superalgebra $\fg$ of type I-0.

We let $M_\la$ be the Verma module of highest weight $\la$ (with respect to~$\le$ in \ref{typeI0})
$$M_\la := U(\mf g)\otimes_{U(\mf b)} \Pi^{\pi(\lambda)} \mC_\la\cong \text{Ind}_{\mf g_{\geq 0}}^{\mf g}M^\oa_\la,$$
Using the element $H_0\in\fh$ and the classical arguments, see e.g.~\cite[\S 1.2]{Hu08}, one shows that~$M_\la$ has a unique maximal submodule. The corresponding unique simple quotient of~$M_\la$ is denoted by~$L_\la$. We let $P_\la\in \mc O$  be the projective cover of~$L_\la$ in~$\mc O$. We will freely use that~$\text{ch} M$ determines completely the Jordan-H\"older decomposition multiplicities $[M:L_\lambda]$.

A module $M$ has a Verma flag if for some $k\in\N$ it has a filtration by submodules
$$0=M_0\subset M_1\subset\cdots\subset M_k=M,$$
where $M_i/M_{i-1}$ is a Verma module for~$1\le i\le k$. We denote by~$(M:M_\lambda)$ the number of indices $i$ for which $M_i/M_{i-1}\cong M_\lambda$. It follows again that the numbers $(M:M_\lambda)$ are determined by~$\text{ch}M$, so in particular do not depend on the chosen filtration. We denote by~$\cO^\Delta$ the full subcategory of modules which have a Verma flag. Furthermore, for any subset $T\subset\fh^\ast$, we denote by~$\cO^\Delta(T)$ the category of modules with Verma flag such that~$(M:M_\lambda)$ is only non-zero for~$\lambda\in T$.

 We now consider the existence and structure of projective covers in~$\mc O$.  For a given weight $\la \in \h^*$, let $Q_\la$ be the projective cover of~$\text{Res}L_\la$ in~$\mc O^\oa$. Therefore we have a $\mf g$-epimorphism $\text{Ind}Q_\la \twoheadrightarrow L_\la$. Observe that   $\text{Ind}Q_\la$ is projective and $\text{Ind}Q_\la = \text{Ind}^{\mf g}_{\mf g_{\geq 0}} \text{Ind}^{\mf g_{\geq 0}}_{\mf g_\oo} Q_\la$. Since $\text{Ind}_{\mf g_\oo}^{\mf g_{\geq 0}} Q_\la$ has a filtration with quotients of~$\mf g_\oo$-Verma modules with trivial $\mf g_1$-action, we may conclude that~$\text{Ind} Q_\la$ has $\mf g$-Verma flag. Consequently, as a direct summand of~$\text{Ind} Q_\la$, the projective cover $P_\la$ of~$L_\la$ has a $\mf g$-Verma flag, see e.g.~\cite[Proposition~3.7]{Hu08}.

Let $(M^\oa_\la)^{\vee}$ denote the dual Verma module in~$\mc O^\oo$, that is,  $(M^\oa_\la)^{\vee}$ is the image of~$M^\oa_\la$ under the duality functor $(\cdot)^\vee$ of~$\mc O^\oo$ from Example~\ref{ExInv}(b). Then we define
$$M_\la^{\vee}:= \text{Coind}_{\mf g_{\leq 0}}^{\mf g} ((M^\oa_\la)^{\vee}).$$

We extend $L_\la^\oa$ trivially to a $\mf {g}_{\geq 0}$-module concentrated in degree $\pi(\lambda)$ and define the (dual) Kac module of~$L_\la^\oa$ as follows: 
$$K_\la :=
\text{Ind}_{\mf g_{\geq 0}}^{\mathfrak{g}}L_\la^\oa\quad\mbox{and}\quad K^{\vee}_\la: = \text{Coind}_{\mf g_{\leq 0}}^{\mf g}L_\la^\oa.$$

Note that it follows from the definitions that we have
\begin{equation}
\label{topsoc}
\Soc (M^\vee_\lambda)\;=\;\Soc (K_\lambda^\vee)\;=\; L_\lambda \;=\; \Top(K_\lambda)\;=\;\Top (M_\lambda).
\end{equation}

Note that \cite[Theorem~2.2]{BF} implies that 
$$K_\lambda\cong \Coind_{\fg_{\ge 0}}^{\fg}L_{\lambda+2\rho_{-1}}^{\oa}\quad\mbox{and}\quad K^\vee_\lambda\cong \Ind_{\fg_{\le 0}}^{\fg}L_{\lambda-2\rho_1}^{\oa}.$$
It thus follows that~$K_\lambda$ also has simple socle and $K^\vee_\lambda$ has simple top. Moreover, calculating the homomorphisms between both modules shows
\begin{equation}
\label{topsoc2}
\Top (K_\lambda^\vee)\;\cong\; \Soc (K_{\lambda-2\omega}).
\end{equation}

\subsubsection{The periplectic Lie superalgebra} \label{subsect::pen} 
We are interested in the {\em periplectic Lie superalgebra}~$\mf{pe}(n)$, which is a subalgebra of the 
general linear Lie superalgebra $\mf{gl}(n|n)$ preserving a non-degenerated odd symmetric bilinear form. We refer the reader to  \cite[Section 1.1]{ChWa12} for more details. 
The standard matrix realisation is given
by
\[ \mf{pe}(n)=
\left\{ \left( \begin{array}{cc} a & b\\
c & -a^t\\
\end{array} \right)\| ~ a,b,c\in \C^{n\times n},~\text{$b$ is symmetric and $c$ is skew-symmetric} \right\}.
\]

Let $E_{ij}$ denote the
$(i,j)$-th matrix unit in~$\mf{gl}(n|n)$, for~$1\le i,j\le 2n$. We set $e_{ij}: = E_{ij} -E_{n+j,n+i}\in \mf{pe}(n)_\oa$,  for all $1\leq i,j \leq n$. 
The Cartan subalgebra $\mf h\subset \mf{pe}(n)_\oa$ is defined as $\mf h: = \bigoplus_{1\leq i \leq n}\C e_{ii}$. Let $\{\vare_i|i=1,\ldots,n\}$ be the dual basis of~$\{e_{ii}\}_{1\leq i \leq n}$ in 
$\mathfrak{h}^*$. 
The algebra $\mf{pe}(n)$ admits a $\mathbb{Z}_2$-compatible $\mathbb{Z}$-gradation inherited from the $\mathbb{Z}$-gradation of~$\mf{gl}(n|n)$ such that 
$$\Phi_{\ob}^-=\{-\vare_i - \vare_j|\text{ } 1\leq i < j \leq n\}\quad\Phi_{\oa}=\{ \vare_i - \vare_j| \text{ } 1\leq i \neq j \leq n\}$$
$$\mbox{and}\quad \Phi_{\ob}^+=\{\vare_i + \vare_j| \text{ } 1\leq i \leq j \leq n\}.$$
Note that~$\mathfrak{pe}(n)$ is of type I-0 for grading element
$H_0\;:=\; \frac{1}{2}\sum_{i=1}^n e_{ii}.$

\section{Harish-Chandra bimodules}\label{SecHC}

In this section, we let $\fg$ be an arbitrary classical Lie superalgebra.

\subsection{Conventions for bimodules}

 Let $A,B$ be two algebras in the set $\{U,U_0\}$ and set $C=U$ if $A=B=U$ and $C=U_0$ otherwise. By $(A,B)$-modules, we mean modules over $A\otimes_{\mC}B^{\op}$. For such a bimodule $N$, we denote by~$N^{\ad}$ the $C$-module obtained via the algebra morphism $C\hookrightarrow A\otimes B^{\op}$ given by the composition of the comultiplication $C\to C\otimes C$ with~$(\id, t)$.
 
We have the left exact functor 
$$\cL(-,-): (B\mbox{-mod})^{\op}\times A\mbox{-mod}\;\to\;A\otimes_{\mC}B^{\op}\mbox{-mod},$$ which takes
the maximal $(A,B)$-submodule of~$\Hom_{\mC}(M,N)$ such that~$\Hom_{\mC}(M,N)^{\ad}$ is a (possibly infinite) direct sum of modules in~$\cF$ if $C=U$, or in~$\cF^{\oa}$ if $C=U_0$. With slight abuse of notation, we will use the same notation $\cL$ for the functor corresponding to each choice of~$A$ and $B$.

For an $(A,B)$-module $M$, we denote by~$\LAnn_A M$ the ideal in~$A$ of elements $u\in A$ such that~$u\otimes 1$ acts trivially on $M$. 

For a (left) $A$-module $V$, a right $B$-module $W$ and an $(A,B)$-module $M$, we interpret $V\otimes_{\mC}M$ and $M\otimes_{\mC}W$ as $(A,B)$-modules in the natural way.

\subsection{$(U,U_0)$-modules}
Let $\cB$ denote the category of finitely generated $(U, U_0)$-modules $N$ for which $N^{\ad}$ is a (possibly infinite) direct sum of modules in~$\cF^{\oa}$.  For a two-sided ideal $J\subset U_0$, we let $\cB(J)$ denote the full subcategory of~$\cB$ of bimodules $X$ such that~$XJ=0$. For any $U_0$-module $M$, we have a canonical monomorphism
\begin{align}
&\iota_M:U_0/\Ann_{U_0}(M)\hookrightarrow \cL(M,M). \label{KSProb}
\end{align}

The following is a variation on~\cite[Theorem~3.1]{MS}. We actually start from a $\fg_{\oa}$-module satisfying the same conditions as {\it loc. cit.} 
\begin{thm}\label{ThmMS1}
Take a $\fg_{\oa}$-module $M$ and set $I:=\Ann_{U_0}(M)$. If
\begin{enumerate}[(a)]
\item the monomorphism $\iota_M$ is an isomorphism, and \label{ThmMS1a}
\item the module $M$ is projective in~$\langle \cF^{\oa}\otimes M\rangle$, \label{ThmMS1b}
\end{enumerate}
then we have an equivalence of categories
$$\Psi=-\otimes_{U_0}M\,:\;\;\; \cB(I)\,\to\, \coker(\cF\otimes \Ind M),$$
with inverse $\cL(M,-)$.
\end{thm}
\begin{proof}
First we will identify the projective modules in~$\cB(I)$. Subsequently, we study the action of~$\Psi$ on the category of projective modules  in~$\cB(I)$. The result then follows from~\cite{BeGe}.

{\em Projective modules in~$\cB(I)$.} We have the submodule $UI$ of the $(U,U_0)$-module $U$, and write $U_I:=U/UI\cong U\otimes_{U_0}U_0/I$.
 For all $X \in \cB(I)$, we have an isomorphism
\begin{equation}\label{eqHomUU0}\Hom_{U-U_0}(U_I,  X)\;\stackrel{\sim}{\to}\; \Hom_{U_0}(\mC,X^{\ad}),\quad \alpha\mapsto \alpha(1+UI),\end{equation}
where we interpret $\Hom_{U_0}(\mC,X^{\ad})$ as the subspace of~$X$ consisting of invariants for the adjoint $\fg_{\oa}$-action.
	Consequently, using adjunction, it follows that, for~$V\in \cF$, we have that~$V\otimes U_I$ is a projective module in~$\cB(I)$, by
	$$\text{Hom}_{U-U_0}(V\otimes U_I, X)\cong  \text{Hom}_{U_0}(\Res V, X^{\text{ad}}),\qquad\mbox{for all $X \in \cB(I)$.}$$
		For an arbitrary~$X$ in~$\cB(I)$, let $\{m_1,\ldots ,m_{\ell} \} \subseteq X$ be a set of~$U-U_0$-generators for~$X$. Then there exists $N\in\cF^{\oa}$ with~$N \subseteq X^{\text{ad}}$ such that~$\{m_1,\ldots ,m_{\ell} \}\subseteq N$. We take a module $\hat{N}\in\cF$ such that~$\Res\hat{N}$ contains $N$ as a direct summand, for instance $\hat{N}=\Ind N$. Then we have a canonical $U$-$U_0$-epimorphism from $\hat{N}\otimes U_I$ to X. 
	In conclusion, $\cB(I)$ has enough projective objects and the category of projective objects is $\add(\cF\otimes \Ind M)$.

	{\em The functor $\Psi$ on projective modules.} 
	Now we turn to the right exact functor
	$$\Psi=-\otimes_{U_0}M:\;\cB(I)\,\to\, U\mbox{-mod}.$$
	For all $N\in\cF$, we have
	$$\Psi(N\otimes U_I)\;\cong\; N\otimes_{\mC} U_I\otimes_{U_0}M\;\cong\; N\otimes\Ind M.$$
	Now we show that~$\Psi$ acts fully faithfully on the category projective modules in~$\cB(I)$. For this we can construct the following commuting diagram for arbitrary~$E,F\in\cF$:
	     $$\xymatrix{\text{Hom}_{U-U_0}( E\otimes U_I,  F\otimes U_I)\ar[r]^{\Psi}\ar[d]^{\simeq}& \text{Hom}_{U}(E\otimes \Ind M,F\otimes \Ind M)\ar[d]^{\simeq} \\
      \text{Hom}_{U-U_0}(  U_I,  (E^\ast)^t\otimes F\otimes U_I)\ar[r]^{\Psi}\ar[d]^{\simeq}& \text{Hom}_{U}( \Ind M,(E^\ast)^t\otimes F\otimes\Ind M)\ar[d]^{\simeq}\\
     \text{Hom}_{U_0-U_0}( U_0/I, (E^\ast)^t\otimes F\otimes \Lambda\fg_{\ob}\otimes  U_0/I)\ar[r]\ar[d]^{\simeq} &\text{Hom}_{U_0}(M,(E^\ast)^t\otimes F\otimes\Lambda\fg_{\ob}\otimes M)\ar[d]^{\simeq}\\
     \text{Hom}_{U_0}( \mC, (E^\ast)^t\otimes F\otimes \Lambda\fg_{\ob}\otimes  (U_0/I)^{\ad})\ar[r] &\text{Hom}_{U_0}(\mC,(E^\ast)^t\otimes F\otimes\Lambda\fg_{\ob}\otimes \cL(M,M)^{\ad})\\
     }$$
	The composition of the lower two vertical arrows on the left is equation~\eqref{eqHomUU0}. The isomorphism in the lowest vertical arrow on the right follows from the fact that~$M$ is the only module in the equation which might not be finite dimensional. We leave as an exercise that the diagram commutes and that the lowest horizontal arrow is an isomorphism as a consequence of assumption (a).

	{\em Application of~\cite[Proposition~5.10]{BeGe}.} 
	For any $N\in\cF$, the functor
	$$\Hom_{U}(N\otimes\Ind M,-)\;\cong\;\Hom_{U_0}(M,\Res (N^* \otimes -))\;: \langle \cF\otimes\Ind M\rangle\to \mC\mbox{-mod}$$
	is exact by assumption (b). Hence the direct summands of modules in~$\cF\otimes \Ind M$ are projective in~$\langle \cF\otimes\Ind M\rangle$. 
	The previous paragraph of the proof also shows that~$\Psi$ is actually a functor between the abelian categories $\cB(I)$ and $\langle \cF\otimes\Ind M\rangle$. (Note that the image of~$\Psi$ is actually even contained in~$\coker(\cF\otimes\Ind M)$.) It thus follows from \cite[Proposition~5.10]{BeGe} that~$\Psi$ yields an equivalence between $\cB(I)$ and the category of modules in~$\langle \cF\otimes\Ind M\rangle$ presented by modules in~$\add(\cF\otimes \Ind M)$, where the latter category is by definition  $\coker(\cF\otimes\Ind M)$.
	To conclude the proof we observe that it follows easily from ordinary bimodule adjunction that~$\cL(M,-)$ is right adjoint to~$\Psi$.   \end{proof}

    Now we choose a Cartan and Borel subalgebra $\fh_{\oa}\subset\fb_{\oa}$ in~$\fg_{\oa}$ and consider the corresponding BGG category~$\cO$ of~$\fg$-modules. For the remainder of this subsection, we fix a regular and dominant weight $\lambda\in\fh_{\oa}^\ast$ and set $\Lambda=\lambda+\Upsilon$. Denote by~$I_\lambda\subset U_0$ the ideal generated by the maximal ideal $m_\lambda=\Ann_{Z(\fg_{\oa})}M^{\oa}_\lambda$ in the centre~$Z(\fg_{\oa})$. Hence
     \begin{equation}\label{AnnEqDuflo}\Ann_{U_0}M^{\oa}_{w\cdot\lambda}\;=\;I_\lambda,\qquad\mbox{for all $w\in W$},
     \end{equation}  see e.g.~\cite[Theorem~10.6]{Hu08}. We set $\cB_\lambda=\cB(I_\lambda)$.

    \begin{cor} \label{CorEqiv2}
    We have mutually inverse equivalences $-\otimes_{U_0}M^{\oa}_\lambda$ and $\cL(M^{\oa}_\lambda,-)$ between 
    $\cB_\lambda$ and $\cO_\Lambda.$
    \end{cor}
    \begin{proof} 
    	The proof will be an application of Theorem \ref{ThmMS1} with~$M = M_\la^\oa$. We first recall that the monomorphism $\iota_{M_\la^\oa}$ in \eqref{KSProb} is always an isomorphism for~$\mf g_\oa$-Verma module $M_\la^\oa$, see e.g. \cite[Section~13.4]{Hu08}. Therefore condition \eqref{ThmMS1a} in Theorem \ref{ThmMS1} is satisfied. 
    	
    	Since $\mc O$ is closed under tensoring with finite-dimensional modules, we have $\langle \cF^{\oa}\otimes M_\la^\oo\rangle \subseteq \mc O_{\Lambda}$. Therefore condition \eqref{ThmMS1b} of Theorem \ref{ThmMS1} is satisfied. As a consequence of Theorem \ref{ThmMS1}, we have an equivalence 
	$$-\otimes_{U_0}M^{\oa}_\lambda:\;\cB_\lambda\,\stackrel{\sim}{\to}\, \coker(\cF\otimes \Ind M_\la^\oa)$$ with inverse $\cL(M^{\oa}_\lambda,-).$
    	By Corollary~\ref{ReDomi}(i), we have $\coker(\cF\otimes \text{Ind} M^\oa_\lambda)=\cO_{\Lambda}$.
    \end{proof}

 \begin{cor}\label{sameAnn} For simple module $L\in\cO_{\Lambda}$, we have $\emph{Ann}_U L = \emph{LAnn}_U \mc L(M_\la^\oa, L).$
 \end{cor}
 \begin{proof}
 Mutatis mutandis \cite[Corollary 4.4(1)]{Co16}.
 \end{proof}
 
 The ideas behind the following two results go back to Vogan, see \cite{Vogan}.   
    
    \begin{lem} \label{LemEbU}
    	 Let $S \in \cB_\lambda$ be a simple bimodule. Then there is $V\in \mc F_\oa$ such that we have a monomorphism of~$(U,U_0)$-modules $$U/\emph{LAnn}_{U}S \hookrightarrow S \otimes V^\ast.$$
    \end{lem}
\begin{proof} 
By Corollary \ref{CorEqiv2}, there exists a simple module $L\in\cO$ such that~$S = \mc L(M_\la^\oa , L)$.
	By \cite[Theorem~3.3]{BeGe}, there exists $V\in \cF^{\oa}$ such that we have an epimorphism
	$$\sigma: V\otimes M_\lambda^{\oa}\tto \Res L$$
of~$\fg_{\oa}$-modules. When we interpret $\sigma\in \Hom_{\mC}(V\otimes M_\lambda^{\oa}, L)$, where the latter space is a $(U,U_0)$-module, $\sigma$ is $\ad_{\fg_{\oa}}$-invariant, so in particular an element of~$\cL(V\otimes M_\lambda^{\oa}, L)$. Moreover, it follows easily that we have a morphism of~$(U,U_0)$-modules
$$\Sigma:\,U\to \cL(V\otimes M_\lambda^{\oa}, L),\quad u\mapsto u\sigma. $$
Clearly, $\Ann_U L$ is contained in the kernel of~$\Sigma$. Since $\sigma$ is an epimorphism, it follows that the kernel is precisely $\Ann_U L$.
Hence we find a monomorphism of~$(U,U_0)$-modules
$$U/\LAnn_U S\;\stackrel{\sim}{\to}\; U/\Ann_U L \;\stackrel{\Sigma}{\hookrightarrow}\; \cL(V\otimes M_\lambda^{\oa}, L)\;\stackrel{\sim}{\to}\;S\otimes V^\ast,$$
where the left isomorphism is Corollary~\ref{sameAnn} and the right one is \cite[Theorem~6.1]{BeGe}.
\end{proof}
    \begin{prop} \label{Prim1}
    Consider two simple objects $S_1$ and $S_2$ in~$\cB_\lambda$. We have $\LAnn_U S_1\subset  \LAnn_US_2$ if and only if there exists a $V \in \cF_\oa$ such that~$S_2$ is a subquotient
of~$S_1\otimes V^\ast $.
    \end{prop}
     \begin{proof} 
     	First we assume that~$[S_1\otimes V^\ast : S_2] \neq 0$, for some $V\in \mc F_\oa$. Then 
     	$$\text{LAnn}_US_1 = \text{LAnn}_U(S_1\otimes V^\ast) \subset \text{LAnn}_U S_2,$$ as desired.

     	Next we set $J_i : = \LAnn S_i$ for~$i=1,2$, and assume that~$J_1\subset  J_2$.     	  
     	By Lemma \ref{LemEbU}, there are $V_i\in \mc F_\oa$ such that~$U/J_i \hookrightarrow S_i \otimes V_i^\ast$ for~$i=1,2$. We shall show that~$V:=V_1 \otimes (V_2^\ast)^t$ is the desired $V\in \cF^{\oa}$.
     	By the fact that~$J_1\subset  J_2$ and $U/J_2 \hookrightarrow S_2 \otimes V_2^\ast$, we have $$0\neq \text{Hom}(U/J_2, S_2\otimes V_2^\ast) \subset \text{Hom}(U/J_1, S_2\otimes V_2^\ast)\cong \text{Hom}(U/J_1 \otimes V_2^t, S_2),$$
     	which implies that~$S_2$ is a quotient of~$U/J_1 \otimes V_2^t$. Finally, by the fact that~$U/J_1 \hookrightarrow S_1 \otimes V_1^\ast$ we may conclude that
        $$U_1/J_1 \otimes V_2^t\hookrightarrow S_1\otimes (V_1^\ast\otimes V_2^t).$$ As a consequence, $S_2$ is indeed a subquotient of~$S_1\otimes V^{\ast}$.     	
    \end{proof}

\subsection{$U$-bimodules}
We denote by~$\cH$ the category of finitely generated~$U$-bimodules $X$, such that~$X^{\ad}$ is a direct sum of modules in~$\cF$. For a two-sided ideal $J\subset U$, we let $\cH(J)$ denote the full subcategory of~$\cH$ of bimodules $X$ such that~$XJ=0$. For any $U$-module $N$, we have a monomorphism
$$\iota_N:U/\Ann_U(N)\hookrightarrow \cL(N,N).$$

The following is a variation on~\cite[Theorem~3.1]{MS}. The difference with the statement {\it loc. cit.} is that~$M$ itself will not necessarily be projective in~$\coker(\cP\otimes M)$. An explicit example of this will be considered in Proposition~\ref{PropHCperi}. Of course, when $\fg=\fg_{\oa}$, we have $\cF=\cP$ and our result reduces to~\cite[Theorem~3.1]{MS}. 
\begin{thm}\label{ThmMS}
Take a $\fg$-module $M$ and set $I:=\Ann_U(M)$. If
\begin{enumerate}[(a)]
\item the monomorphism $\iota_M$ is an isomorphism, and
\item all modules in~$\cP\otimes M$ are projective in~$\langle \cP\otimes M\rangle$,
\end{enumerate}
then we have an equivalence of categories
$$\Theta=-\otimes_UM\,:\; \cH(I)\,\to\, \coker(\cP\otimes M),$$
with inverse $\cL(M,-)$.
\end{thm}
\begin{proof}
We first identify the projective modules in~$\cH(I)$ with direct summands of modules in~$\cP\otimes U/I$.
	If $V\in \cP$, then $V\otimes U/I$ is a projective module in~$\cH(I)$ since we have
	$$\text{Hom}_{U^2}(V\otimes U/I, X) =\text{Hom}_{U}(V, X^{\text{ad}}),\qquad\mbox{for all $X \in \cH(I)$.}$$
	Now we assume that~$X$ is a projective module in~$\cH(I)$. Let $\{m_1,\ldots ,m_{\ell} \} \subseteq X$ be a set of~$U^2$-generators for~$X$. Then there exists a finite-dimensional $\fg$-submodule $N \subseteq X^{\text{ad}}$ such that~$\{m_1,\ldots ,m_{\ell} \}\subseteq N$. Therefore we have a canonical $U^2$-homomorphism from $N\otimes U/I$ to~$X$ sending  $n\otimes u$ to~$nu$ for each $n\in N$ and $u \in U/I$. Consequently, $X$ is an epimorphic image of~$N\otimes U/I$. 
	Let $\hat N\in \mc P$ be the projective cover of~$N$ in~$\mc F$, then we have $\hat N\otimes U/I \twoheadrightarrow N\otimes U/I$. Therefore~$X$ is a direct summand of~$\hat N\otimes U/I$, as desired. 
	
	In particular it now follows that~$\Theta$ maps projective objects in~$\cH(I)$ to direct summands of modules in~$\cP\otimes M$. We can then proceed as in the proof of Theorem~\ref{ThmMS1}.
	\end{proof}

Now we consider the specific case of the periplectic Lie superalgebra $\fg=\mathfrak{pe}(n)$ and conclude this section by using Theorem~\ref{ThmMS} to provide an equivalence between category~$\cO$ for the periplectic Lie superalgebra and a category of~$\fg$-$\fg$ Harish-Chandra bimodules. For all other simple classical Lie superalgebras, such an equivalence was already obtained in~\cite[Theorem~4.1]{MaMe12}.
We use the notion of typicality for~$\mf{pe}(n)$ defined in~\cite[Section 5]{Se02}. Let $\la \in \mf h^*$ be a typical, dominant and regular weight, and set
$$\Lambda=\lambda+\Upsilon, \quad J_\lambda=\Ann_U M_\la,\quad \ov{\mc H}_\la= \cH(J_\lambda)\quad\mbox{and}\quad \ov{U}_\la = U_\la /J_\la.$$ We use the notation $\ov{\mc H}_\la$, rather than $\cH_\lambda$ to stress that we do not just impose a condition on the central character of the right action.

\begin{prop}\label{PropHCperi}
	The categories $ \mc O_{\Lambda}$ and $\ov{\mc H}_\la$ are equivalent. Mutually inverse equivalences are given by~$-\otimes_{\ov U_\la}M_\la$ and $\mc L(M_\la, -)$.
\end{prop} \begin{proof}
	Observe that the category~$\cP\otimes M_\la$, and hence also $\langle \cP\otimes M_\la\rangle $, is contained in~$\cO_{\Lambda}$.
	By Corollary~\ref{ReDomi}(ii), the category of projective modules in~$\mc O_\Lambda$ is $\add(\cP\otimes M_\la)$. It follows in particular that
	$\coker(\cP\otimes M_\lambda)=\cO_{\Lambda}.$
	By  \cite[Proof of Theorem 5.7]{Se02}, the canonical monomorphism $\Ula \rightarrow \mc L(M_\la,M_\la)$ is an isomorphism. The claim thus follows from Theorem~\ref{ThmMS}.
\end{proof}

\section{Completion functors and the primitive spectrum}\label{SecPrim}

Duflo proved that every primitive ideal of a semisimple Lie algebra is an annihilator of a simple module in BGG category~$\cO$. In \cite{Mu92}, Musson proved, by building on work on finite ring extensions in~\cite{Le89}, that the analogue of this statement remains true for simple classical Lie superalgebras.
In Section~\ref{SecMus}, we observe that the methods in~\cite{Le89, Mu92} actually show that Duflo's result remains valid for arbitrary (not necessarily simple) classical Lie superalgebras.

To describe the primitive spectrum of~$U(\fg)$, one is thus left with the problem of determining all inclusions between annihilator ideals of simple modules in category~$\cO$. 
In Section~\ref{SecCom}, we will extend a result in~\cite{Co16} for {\em simple basic} classical Lie superalgebras, describing the primitive spectrum in terms of completion functors, to {\em all} classical Lie superalgebras. 
In Section~\ref{SecTwi1} we make the connection with twisting functors.


\subsection{Primitive ideals for classical Lie superalgebras} \label{SecMus} For a Lie superalgebra $\fg$, we refer to the annihilator ideals in~$U(\fg)$ of the simple ($\mZ_2$-graded) modules as its primitive ideals.
The following result, which holds for an arbitrary choice of Borel subalgebra $\fb_{\oa}\subset \fg_{\oa}$, is an immediate generalisation of Musson's theorem of~\cite[Section 2.2]{Mu92}. 
\begin{thm}\label{ThmPrim} Let $\mf g$ be a classical Lie superalgebra. Then every primitive ideal is an annihilator of a simple module in~$\cO$.
\end{thm}

We start by formulating some results of Letzter (see e.g.~\cite{Le89}) in the form that we will need. We will use the survey in~\cite[Section~7.6]{Mu12} as reference.
By `algebra', we mean a unital associative $\mZ_2$-graded algebra over $\mC$. By definition, the primitive ideals of an algebra $R$ are the annihilator ideals of the simple (left) modules. We say a set~$\bL$ of simple $R$-modules is Ann-complete if every primitive ideal is of the form $\Ann_RL$ for some $L\in \bL$. 

We will consider algebras $R,S$ with the following properties (see~\cite[Hypothesis~7.1.1]{Mu12}):
\begin{enumerate}[(a)]
\item $R$ is a finitely generated noetherian algebra of finite GK dimension;
\item $S$ contains $R$ (as graded subalgebra) and is finitely and freely generated as a left $R$-module;
\item $R$ is a direct summand of the right $R$-module $S$.
\end{enumerate}

\begin{prop}\label{PropLM}\cite{Le89, Mu12}
Consider algebras $R,S$ satisfying (a)-(c). Assume that we have an $\Ann$-complete set $\bL_R$ for~$R$ such that the $S$-module $\Ind^S_RL$ has finite length for all $L\in\bL_R$. Then the set
$$\bL_S:=\{K\mbox{ is a simple $S$-module with } [\Ind^S_RL:K]\not=0\mbox{ for some $L\in\bL_R$} \},$$
is an $\Ann$-complete set for~$S$.
\end{prop}
\begin{proof}
For an algebra $T$ we let $\Spec T$ denote the set of all (graded) prime ideals and $\Prim R\subset\Spec T$ the set of all primitive ideals. We note that in~\cite[Section~7]{Mu12} the notation GrSpec and GrPrim are used.
By \cite[Corollary~7.6.14]{Mu12}, we have
$$\Prim S\;=\; \bigcup_{I\in \Prim R}\{J\in \Spec S\,|\, J\mbox{ is minimal over }\Ann_S(S/SI)\}.$$
Here ``$J$ is minimal over $Q$'' for some ideal $Q$ and prime ideal $J$ means that~$J\supseteq Q$ and there is no $J'\in\Spec R$ with~$J\supsetneq J'\supseteq Q$.

Since $\bL_R$ is $\Ann$-complete, \cite[Lemma~7.6.15]{Mu12} implies
$$\Prim S\;=\; \bigcup_{L\in \bL_R}\{J\in \Spec S\,|\, J\mbox{ is minimal over }\Ann_S(\Ind^S_RL)\}.$$

Now let $M$ be an arbitrary~$S$-module with a finite filtration
$$0=M_k\subsetneq M_{k-1}\subsetneq \cdots\subsetneq M_1\subsetneq M_0=M,$$
with~$L_i:=M_{i-1}/M_{i}$ simple for~$1\le i\le k$.
If $J\in \Spec S$ is minimal over $\Ann_SM$, then $J=\Ann_SL_i$ for some $i$. This well-known property can be proved as follows. With $J_i:=\Ann_SL_i$, for all $1\le i\le k$, we have
$$J_{k}J_{k-1}\cdots J_2 J_1\subseteq \Ann_S(M)\subseteq J,$$
where we note that the order of the ideals in the left term is relevant. Since $J$ is prime, we have $J_i\subseteq J$, for some $1\le i\le k$. Since we also have $\Ann_SM\subseteq J_i$, the minimality of~$J$ implies $J=J_i$.

We thus find
$$\Prim S\;=\; \bigcup_{L\in \bL_R}\{\Ann_S K\,|\, K \mbox{ is a simple $S$-module with }\; [\Ind_R^SL:K]\not=0\},$$
which concludes the proof.
\end{proof}

\begin{proof}[Proof of Theorem~\ref{ThmPrim}]
The conditions (a)-(c) are satisfied for~$R=U(\fk_{\oa})$ and $S=U(\fk)$, with~$\fk$ an arbitrary finite dimensional Lie superalgebra $\fk$, see \cite[Corollary~6.4.6]{Mu12}. Moreover, by~\cite{Duflo} the set of simple modules in category~$\cO^{\oa}$ is $\Ann$-complete for~$U_0$. Finally $U(\fg)\otimes_{U(\fg_{\oa})}L$ is of finite length for any simple $L\in\cO^{\oa} $ since even $\Res \Ind L\cong \Lambda\fg_{\ob}\otimes L\in \cO^{\oa}$ is of finite length as a $\fg_{\oa}$-module.
 \end{proof}

 \subsection{Completion functors}\label{SecCom} Consider a classical Lie superalgebra $\fg$. For this subsection, we fix a regular and dominant weight $\lambda\in\fh_{\oa}^\ast$ and set $\Lambda=\lambda+\Upsilon$. 
 We define the Enright completion functor for~$s\in W^{\lambda}$ a simple reflection, see e.g. \cite[Section 2]{Jo82}, as
 $$G_s(-) := \mc L(M^\oa_{s \cdot \la}, -)\otimes_{U_\oa} M_\la^\oa:\; \mc O_{\Lambda} \rightarrow \mc O_{\Lambda}.$$
By \eqref{AnnEqDuflo}, it follows that~$G_s$ is a composition 
 $$\cO_\Lambda\;\to\;\cB_\lambda\;\stackrel{\sim}{\to}\; \cO_\Lambda,$$
 of~$\cL(M^\oa_{s \cdot \la}, -)$ with the equivalence in Corollary~\ref{CorEqiv2}. In particular, $G_s$ is well-defined. By definition, we also have
\begin{equation}\label{ResG}\Res\circ G_s\;\cong\; G_s\circ\Res,\end{equation}
 where we use the same notation $G_s$ for the classical completion functor on $\cO^{\oa}$.
 
 \begin{thm} \label{CoPI}
 Set $J_i:= \emph{Ann}_UL_{i}, $ with~$L_i\in\cO_{\Lambda}$ a simple module, for~$i=1,2$. Then $J_{1}\subseteq J_2$ if and only if $L_{2}$ is a subquotient of~$G_{s_1}G_{s_2}\cdots G_{s_k}L_{1}$, for some simple reflections $s_1, s_2 , \ldots, s_k \in W^{\lambda}$.
 \end{thm}
\begin{proof} We adapt the proof of~\cite[Theorem 5.3]{Co16}. 
	In view of Corollaries \ref{CorEqiv2} and \ref{sameAnn} and Proposition \ref{Prim1}, it suffices to show that~$\mc L(M^\oa_{\la}, L_{2})$ is a subquotient of 
	$$\mc L(V\otimes M_\la^\oa, L_{1}) = \cL( M_\la^\oa, L_{1})\otimes V^* ,$$ for some $V\in \mc F^\oa$ 
 if and only if $\mc L(M^\oa_{\la}, L_{2})$ is a subquotient of~$$\cL(M_\lambda^{\oa},G_{s_1}G_{s_2}\cdots G_{s_k}L_{1}),$$ for some simple reflections $s_1,s_2,\ldots, s_k\in W^{\lambda}$.
 
By definition, projective functors on the category of $U_0$-modules which admit generalised central character are the direct summands of functor of the form $-\otimes V$ for some $V\in\cF^{\oa}$.
We will use the classification result of~\cite[Theorem~3.3]{BeGe}, which states in particular that the (isomorphism classes of) indecomposable projective functors on the block $\cO^{\oa}_\lambda$ containing $M_\lambda^{\oa}$ are in bijection with the set $W^{\lambda}$. We write $\theta_x$, with $x\in W^{\lambda}$, for the corresponding exact functor.
	By \cite[Chapter~7]{Hu08}, for every simple reflection $s\in W^{\lambda}$, we have a short exact sequence
	$$
	0\rightarrow M^\oa_{\la} \rightarrow \theta_{s}M_\la^\oa \rightarrow M_{s\cdot \la}^\oa \rightarrow 0.
	$$
 By applying $\mc L(-,L_{1})$, we have the exact sequence
$$
 0\rightarrow \mc L(M^\oa_{s\cdot \la}, L_{1}) \rightarrow \mc L(\theta_s M_{\la}^\oa, L_{1}) \rightarrow \mc L(M_\la^\oa, L_{1}),
$$
 which we can rewrite, using the equivalences in Corollary~\ref{CorEqiv2}, as
$$
 0\rightarrow \mc L(M^\oa_{ \la}, G_s L_{1}) \rightarrow \mc L(\theta_s M_{\la}^\oa, L_{1}) \rightarrow \mc L(M_\la^\oa, L_{1}).
$$
 Hence, for any simple $X\in \cB_\lambda$ different from $\cL(M_\la^\oa, L_{1})$, we have
 $$[\cL(M^\oa_{ \la}, G_s L_{1}): X]\;=\;[\cL(\theta_s M_{\la}^\oa, L_{1}):X].$$

	Therefore if $\cL(M_\lambda^{\oa},L_{2})$ is a subquotient of~$\cL(M_\lambda^{\oa},G_{s}L_{1} )$, then $\cL(M_\lambda^{\oa},L_{2})$ is a subquotient of~$\mc L(\theta_s M_{\la}^\oa, L_{1})$, which is a direct summand of~$$\mc L(V\otimes M^\oa_{\la}, L_{1}) = \mc L(M_\la^\oa, L_{1}) \otimes  V^* ,$$ for some $V\in\cF^{\oa}$. This proves one direction of the claim.

	Now we assume that~$\mc L(M_\la^\oa, L_{2})$ is a subquotient of~$\mc L(V\otimes M_\la^\oa, L_{1})$ for some $V \in \mc F^\oa$. Since $\mc L( M_\la^\oa, L_{1})\otimes V^\ast$ is itself a direct summand of~$\mc L(M_\la^\oa, L_{1})\theta_{s_1}\theta_{s_2}\cdots \theta_{s_k}$ for some indecomposable projective functors $\theta_{s_i}$ with simple reflections $s_1, s_2, \ldots s_k$, there are mutually different simple objects $L^1 \neq  \mc L(M_\la^\oa, L_{1}), L^2,\ldots , L^{k-1}\neq \mc L(M_\la^\oa, L_{2})$ such that~$$[\mc L(M_\la^\oa, L_{1})\theta_{s_1}:L^1]\neq 0, ~[L^1\theta_{s_2}:L^2]\neq 0, \ldots , [L^{k-1}\theta_{s_k}:\mc L(M_\la^\oa, L_{2})]\neq 0.$$
	Consequently, we have $[G_{s_k}\cdots G_{s_1}L_{1}: L_{2}]\neq 0$ as desired.
\end{proof}

Theorem~\ref{CoPI} describes the inclusions between annihilator ideals of simple modules in~$\cO_\Lambda$, for a fixed $\Lambda\in\fh^\ast/\Upsilon$. In order to describe the relation between simple modules in different such subcategories, now we define the completion functor for a simple reflection $s\in W$:
 $$G_s(-) := \mc L(M^\oa_{s \cdot \la}, -)\otimes_{U_\oa} M_\la^\oa:\; \mc O^{s\cdot\Lambda} \rightarrow \mc O_{\Lambda}.$$
 If we actually have $s\in W^{\lambda}$, then by assumption, $\Lambda=s\cdot\Lambda$ and $s$ is of course also simple as a reflection in~$W^{\lambda}$. Hence $G_s$ is then already studied above and the interesting case is therefore when $s\not\in W^{\lambda}$.
 \begin{prop}\label{PropNotInt}
 When $s\not\in W^{\lambda}$, we have an equivalence $G_s:\cO_{s\cdot \Lambda}\stackrel{\sim}{\to}\cO_{\Lambda}$ which maps simple modules to simple modules with the same annihilator ideal.
 \end{prop} 
 \begin{proof}
 Under the assumptions $s\cdot\lambda$ is also dominant. By Corollary~\ref{CorEqiv2} we thus find that~$G_s$ is a composition of equivalences
 $$\cO_{s\cdot\Lambda}\,\stackrel{\sim}{\to}\, \cB_{s\cdot\lambda}\,=\,\cB_{\lambda}\,\stackrel{\sim}{\to}\,\cO_{\Lambda},$$
 where the middle equation follows from equation~\eqref{AnnEqDuflo}.
 The claim about annihilator ideals follows from applying Corollary~\ref{sameAnn} twice.
 \end{proof}

 \subsection{Twisting functors}\label{SecTwi1} For basic classical Lie superalgebra, it is proved in~\cite[Theorem 5.5]{Co16} that the completion functors are right adjoint to the twisting functors and furthermore these functors are isomorphic up to conjugation with the duality functor of Example~\ref{ExInv}(a). This generalises the corresponding properties for reductive Lie algebras in~\cite[Theorem 4.1]{AS} and \cite[Theorem 3]{KM}.
In this subsection, we derive analogous results for $\mf{pe}(n)$. For this, we will introduce a new duality on the $\mc O$ for~$\mathfrak{pe}(n)$ which however does not preserves simple modules. First we recall the construction of twisting functors of~\cite{Ar}. 
 
  For a root $\beta \in \Phi$, we denote by~$\mf g_{\beta}$ the root space of~$\mf g$ associated with~$\beta$. Fix a simple root $\alpha \in \Phi_\oo^+$ and a non-zero root vector $X \in (\mathfrak{g}_{\bar{0}})_{-\alpha}$. Then we have the Ore localisation $U'_{\alpha}$ of~$U$ with respect to the set of powers of~$X$ since the adjoint action of~$X$ on $\mathfrak{g}$ is nilpotent. Now $X$ is not a zero divisor in~$U$, therefore~$U$ can be viewed as an associative subalgebra of~$U'_{\alpha}$. The quotient $U_\alpha:=U'_{\alpha}/U$ is thus a $U$-$U$-bimodule.
  Let $\varphi=\varphi_{\alpha}$ be an automorphism of~$\mathfrak{g}$ that maps $(\mathfrak{g}_{i})_{\beta}$ to  $(\mathfrak{g}_{i})_{s_{\alpha}(\beta)}$ for all simple roots $\beta$ and $i\in \{\bar{0}, \bar{1}\}$. Now we let $^{\varphi}U_{\alpha}$ be the bimodule obtained from $U_{\alpha}$ by twisting the left action of~$U$ by~$\varphi$. The twisting functor is then defined as
$$
 T_{s_{\alpha}}(-)= T_{\alpha}(-):=  ^{\varphi}U_{\alpha}\otimes - : \mathcal{O} \rightarrow \mathcal{O}.
$$
 We will use the same notation $T_\alpha$ for the twisting functor defined in the same way on $\cO$ for any subalgebra of~$\fg$ containing $\fg_{\oa}$. Also, we define $G_{\alpha}: = G_{s_{\alpha}}.$

 Motivated by Proposition~\ref{PropNotInt} and the corresponding results in~\cite[Section~5]{CoMa14} and \cite[Proposition~3.9]{CMW}, as well as for simplicity, we will restrict to integral blocks of category~$\cO$ for the remainder of this section. 
 
  \begin{thm} \label{TsGs} For $\mf g$ a classical Lie superalgebra of type I-0 with good involution $\sigma$ and simple $\alpha\in\Phi_{\oa}^+$, we have $D_\sigma\circ G_{\sigma^\ast(\alpha)}\circ D_{\sigma} \cong T_{\alpha}$ on $\mc O_\Upsilon $ and $G_{\alpha}$ is right adjoint to~$T_\alpha$.  \end{thm}
\begin{proof}
      We set $\beta:=\sigma^\ast(\alpha)$ and $D=D_\sigma$.
      First, we will interpret $D$ as a functor acting between $\cO(\fg_{\ge 0},\fb_{\oa})$ and $\cO(\sigma(\fg_{\ge 0}),\fb_{\oa})$ and $G_\beta$ as an endofunctor of~$\cO(\sigma(\fg_{\ge 0}),\fb_{\oa})$. Note that~$\fg_{\ge 0}$ and $\sigma(\fg_{\ge 0})$ are also classical Lie superalgebras of type I-0.

   Fix an arbitrary~$\zeta \in \Upsilon$ with~$s_\alpha\cdot\zeta\le \zeta$ and consider the $\fg_{\ge 0}$-module $N_\zeta$, which is the $\fg_{\oa}$-Verma module $M_\zeta^\oo$ with trivial $\fg_{1}$-action. By definition of twisting functors and \cite[Section~2]{AS} we have
 $$\Res_{\mf g_{ 0}}^{\mf g_{\geq 0}}(T_{\alpha}  N_\zeta) \cong  T_{\alpha} (M^\oo_\zeta)\cong M^{\oa}_{s_{\alpha}\cdot\zeta}.$$
 On the other hand, we have by~\cite[Theorem 4.1]{AS} and \cite[Theorem 3]{KM} and equation~\eqref{ResG}
   $$\Res_{\mf g_0}^{\mf g_{\geq 0}} D\circ G_\beta\circ D  (N_\zeta) \cong  (\cdot)^\vee \circ G_{\alpha}\circ (\cdot)^\vee (M^\oo_\zeta)\cong M^{\oa}_{s_{\alpha}\cdot\zeta}.$$
 Since there is only one structure of a $\fg_{\ge 0}$-module on $M^{\oa}_{s_{\alpha} \cdot\zeta}$ which extends the $\fg_{0}$-action, we have
that~$T_\alpha N_\zeta$ and $DG_\beta D N_\zeta$ are isomorphic
  as $\fg_{\ge 0}$-modules.

Now we turn to~$\fg$-modules. We have
$$DG_\beta D(M_\zeta)\cong \Ind^{\fg}_{\fg_{\ge 0}} DG_\beta D (N_\zeta)\cong \Ind^{\fg}_{\fg_{\ge 0}} T_\alpha N_\zeta \cong T_\alpha M_{\zeta}.$$   
Since $\text{dim}\text{End}_{\mf g_{\geq 0}}(M_{\zeta}) =1$, it follows that~$D\circ G_{\beta}\circ D \cong T_{\alpha}$ when restricted to the full subcategory of~$\cO_\Upsilon$ with one module $M_\zeta$.

 By 
 \cite[Lemma~5.9]{CoMa14} twisting functors commute, as functors, with functors of the form $-\otimes V$ with~$V\in\cF$. By construction, the same is true for completion functors. It thus follows that~$D\circ G_{\beta}\circ D $ and $T_{\alpha}$ are isomorphic on the category~$\cF\otimes M_\zeta$. Now we take $\zeta$ dominant and regular. By Corollary~\ref{ReDomi}(ii), we have that~$\add(\cP\otimes M_\zeta)$ is the category of projective modules in~$\cO_{\Upsilon}$.
Hence $D\circ G_{\beta}\circ D $ and $T_{\alpha}$ are isomorphic on the category of projective modules.  Since they are both right exact, they are isomorphic on the entire category~$\cO_{\Upsilon}$.

Similarly, we can prove that the right adjoint of~$T_\alpha$ is isomorphic to~$G_\alpha$ on the category with one object $DM_\zeta$, from which the result on $\cO_\Upsilon$ follows.
\end{proof}

 To apply this result to the periplectic Lie superalgebra, we need some preparation. Define the anti-involution $\sigma$ on $\mf{gl}(n|n)$ by 
 $$\sigma(E_{ij}) : = (-1)^{|i|(|j|+1)}E_{2n+1-j,~ 2n+1-i},\quad\mbox{for~$1\leq i,j \leq 2n$, }$$
 where $|k|=0$ when $k\le n$ and $|k|=1$ when $k>n$.
 
 \begin{lem}\label{Lemspe}
 The anti-involution $\sigma$ restricts to a good involution of~$\fg=\mathfrak{pe}(n)\subset\mathfrak{gl}(n|n)$. Furthermore, we have $\sigma^\ast (\lambda)=-w_0(\lambda)$, for all $\lambda\in\fh^\ast$. In particular we have $$\sigma(\fg_{\alpha})\;=\; \fg_{w_0(\alpha)},\quad \mbox{for all $\alpha\in\Phi$}.$$ 
 \end{lem}
 \begin{proof}
 Let $1\leq i,j \leq n$, and recall that we set $e_{ij}: = E_{ij} -E_{n+j,n+i}\in \mf{pe}(n)_\oa$. Furthermore, we define $e_{ij}^{(1)}: = E_{i,n+j}+E_{j,n+i}\in \mf{pe}(n)_{1}$ and $e_{ij}^{(-1)}: = E_{n+i,j} -E_{n+j,i}\in \mf{pe}(n)_{-1}$. Therefore~$\sigma$ preserves $\mf{pe}(n)$ and restricts to the following anti-automorphism of~$\mf{pe}(n)$:
 \begin{align}
 &\sigma(e_{ij}) = -e_{n+1-i, n+1-j}, ~
 \sigma(e_{ij}^{(k)}) =k e_{n+1-j, n+1-i}^{(k)},
 \end{align}
 for~$1\leq i,j\leq n$ and $k =-1, 1$. The description of~$\sigma^\ast$ follows from $\sigma(e_{ii})=-e_{n+1-i,n+1-i}$. 
 \end{proof}

  Since we have $\mathfrak{pe}(n)_{\oa}=\mathfrak{gl}(n)$, we have the canonical ordering $\{\alpha_i\,|\,1\le i\le n-1\}$ of simple roots in~$\Phi_{\oa}^+$. For simplicity, we write $T_i: = T_{s_i},$ for simple reflection $s_1, s_2, \ldots, s_{n-1}.$ Similarly, we denote completion functors by $G_1, G_2, \ldots , G_{n-1}$. We denote  the longest element in~$ W=S_n$ by~$w_0$ .
 
 As a special case of Theorem~\ref{TsGs}, we find the following.
   \begin{thm}\label{TsGsp} If $\mf g=\mf{pe}(n)$ and $\sigma$ as in Lemma~\ref{Lemspe}. Then $D_\sigma\circ T_{i}\circ D_\sigma \cong G_{n-i}$ on $\mc O_\Upsilon $ and $G_i$ is right adjoint to~$T_i$, for each $i=1,2,\ldots ,n-1$.
  \end{thm}
 
The following is a consequence of Theorems \ref{CoPI} and \ref{TsGsp}:

 \begin{cor} \label{TwPI} Consider $\mf g = \mf{pe}(n)$. 
	Let $\la_1, \la_2 \in \Upsilon$ and set $J_i:= \emph{Ann}_UL_{\la_i}, $ for~$i=1,2$. Then $J_{1}\subseteq J_2$ if and only if $D_\sigma L_{\la_2}$ is a subquotient of~$T_{s_1}T_{s_2}\cdots T_{s_k}D_\sigma L_{\la_1}$, for some simple reflections $s_1, s_2 , \ldots, s_k \in W$.
\end{cor}

\section{Category~$\cO$ for the periplectic Lie superalgebra}\label{SecO}

In this section, we study the BGG category~$\mc O$ over $\mf g = \mf{pe}(n)$. In fact, we will work with~$\cO_{red}$ for an unspecified choice of~$\pi:\fh^\ast\to\mZ_2$ as in Section \ref{SecDefO}, but simply write $\cO$. We use all notation and conventions from Section~\ref{SecPrel}.
The values $\rho_1,\rho_{-1},\omega$ introduced in Section \ref{typeI0} are given by
$$\omega : = \vare_1 +\vare_2 +\cdots +\vare_n,\quad\rho_1 =  \frac{(1+n)}{2}\omega\quad\text{and}\quad\rho_{-1} = \frac{(1-n)}{2}\omega.$$

\subsection{BGG reciprocity} 


\begin{lem} \label{BGGlem}
	If $N \in \mc O$ has a Verma flag, we have
	$$(N: M_\mu) = \emph{dimHom}_{\mc O}(N, M_{\mu}^{\vee}).$$

\end{lem}
\begin{proof}
		For given $\la, \mu\in \h^*$ and $i\ge 0$, we have the following calculation
	\begin{align*}
	&\text{Ext}^i_{\mf g}(M_\la, M_{\mu}^{\vee})\cong \text{Ext}^i_{\mf g_{\leq 0}}(\text{Res}_{\mf g_{\leq 0}}^{\mf g}M_\la, (M^\oa_\mu)^{\vee}) \\
	&\cong \text{Ext}^i_{\mf g_{\leq 0}}(\text{Ind}^{\mf g_{\leq 0}}_{\mf g_0}M^\oa_\la, (M^\oa_\mu)^{\vee}) \hskip 0.5cm \text{ (since $\text{Res}_{\mf g_{\leq 0}}^{\mf g}\text{Ind}^{\mf g}_{\mf g_{\geq 0}}M_\la^\oa \cong \text{Ind}^{\mf g_{\leq 0}}_{\mf g_0}M^\oa_\la$)} \\
	&\cong \text{Ext}^i_{\mf g_{0}}(M^\oa_\la, (M^\oa_\mu)^{\vee}) \hskip 0.5cm \\
	& \cong \left\{ \begin{array}{ll} \C,\quad \text{if $\la = \mu$ and $i=0$,} \\
	0, \quad\text{otherwise},\end{array} \right.
	\end{align*} 
	where the last isomorphism is \cite[Theorems 3.3(d) and 6.12]{Hu08}.
	The conclusion then follows from induction on the length of Verma flags. 
\end{proof}

We have the following relation between characters of Verma and dual Verma modules.
\begin{lem} \label{LemChVInDV}
Let $S\subset\fh^\ast$ be the set of weights $\sum_{i=1}^n a_i\varepsilon_i$, where $a_i\in\{0,2\}$.
We have 
$$\emph{ch}M^{\vee}_\mu =\sum_{\kappa\in S}\emph{ch}M_{\mu-\kappa}=\sum_{\kappa\in S}\emph{ch}M_{\mu-2\omega+\kappa},\quad\mbox{for all $\mu\in\fh^\ast$}.$$\end{lem}
\begin{proof}
We have
$$\text{ch}\Lambda \mf g_1^\ast = \prod_{1\leq i\le j\leq n}(1+ e^{-\vare_i-\vare_j})\quad \mbox{and}\quad \text{ch}\Lambda \mf g_{-1} = \prod_{1\leq i<j\leq n}(1+ e^{-\vare_i-\vare_j}),$$
from which the observation follows. 
\end{proof}

Recall that~$P_\la\in \mc O$ denote the projective cover of~$L_\la$ in~$\mc O$, for~$\la \in \h^*$.
  \begin{prop}[BGG reciprocity] \label{Thm::BGGrcp}
	For given $\la,\mu\in \mf h^*$, we have \begin{align}
	&(P_{\la}: M_\mu) = [M_{\mu}^{\vee}: L_\la].  \label{BGG1} 
	\end{align}
	Also, we have the following characters of projective covers
	\begin{align}
	&[P_\la: L_\mu] = \sum_{\zeta \in \h^*} [M^{\vee}_{\zeta}: L_\la][M_{\zeta}: L_\mu] = \sum_{\zeta \in \h^*} [ M_{\zeta}: L_\la][M^{\vee}_{\zeta+2\omega}: L_\mu]. \label{BGG2} \end{align}
\end{prop}
\begin{proof} 
Equation \eqref{BGG1} follows from Lemma \ref{BGGlem} as
	$$[M^{\vee}_\mu: L_\la]=\text{dimHom}_{\mc O}(P_\la, M^{\vee}_\mu) = (P_\la: M_{\mu}).$$	
	This also implies the first equation in \eqref{BGG2}. The first equation in Lemma~\ref{LemChVInDV} then implies
\begin{equation}\label{BGG3}
[P_\la: L_\mu] \;=\; \sum_{\kappa\in S,\zeta\in\fh^\ast}[M_{\zeta-\kappa}:L_\lambda][M_\zeta:L_{\mu}]\;=\; \sum_{\kappa\in S,\zeta\in\fh^\ast}[M_{\zeta}:L_\lambda][M_{\zeta+\kappa}:L_{\mu}].
\end{equation}
	 The second equation in \eqref{BGG2} then follows from the second equation in Lemma~\ref{LemChVInDV}.
\end{proof}

\subsection{The block decomposition of~$\cO$}

We define the equivalence relation $\sim$ on $\h^*$ which is transitively generated by 
$$\begin{cases}\la\sim \la \pm2\vare_k, &\mbox{ for~$1\le k\le n$;}\\
\la\sim w\cdot \la,&\mbox{ for~$w\in W^\lambda$.}
\end{cases}$$
For any $\lambda\in\fh^\ast$, we denote by~$[\lambda]$ its corresponding equivalence class in~$\fh^\ast/\sim$. Clearly, we have $[\lambda]\subset\lambda+\Gamma$.

\begin{thm}\label{ThmBlock}
\begin{enumerate}[(i)]
\item The simple modules $L_\mu$ and $L_\lambda$ are in the same block if and only if $\lambda\sim\mu$. Consequently, we have a decomposition
$$\cO\;\cong\;\bigoplus_{\xi\in\fh^\ast/\sim}\cO_\xi,$$
where $\cO_\xi$ is the Serre subcategory generated by the simple modules $\{L_\lambda\,|\,\lambda\in\xi\}$. 
\item For $\mc O_{\Z}:=\mc O_{\oplus_{i=1}^n\Z \vare_i}$ the full subcategory of~$\mc O$ consisting of modules with integer weights, we have 
$$\cO_{\mZ}\;\cong\;\bigoplus_{i=0}^n \cO_{[\partial^i]},$$
with~$\partial^{i}:=i\vare_1+(i-1)\vare_2 + \cdots + \vare_{i}$ appearing in~\cite[Section 7.1.1]{Co16}. 
\end{enumerate}
\end{thm}

Before proving the theorem, we mention the following lemma which shows that~$\sim$ is an analogue of the equivalence relation defined in~\cite[Definition~5.1]{Ch} and \cite[Section 8.3]{Co16}.
\begin{lem}
	Let $\la \in \h^*$ with~$\la_k = \la_{k+1}$ for some $1\leq k <n$. Then $\la \sim \la +(\vare_k +\vare_{k+1})$.
\end{lem}
\begin{proof}
   We may note that 
   $$\la \sim s_{\vare_k- \vare_{k+1}}\cdot \la \sim (s_{\vare_k- \vare_{k+1}}\cdot \la) + 2 \vare_k =  \la +(\vare_k +\vare_{k+1}),$$
   as desired.
\end{proof}

\begin{prop} \label{Prop::Linkage}
For $\la, \mu \in \h^*$ with~$\la \sim \mu$, we have that~$L_\la$ and $L_{\mu}$ lie in the same block.
\end{prop}
\begin{proof}
By definition of~$\sim$, it suffices to prove the following, for arbitrary~$\la \in \h^*$.
\begin{enumerate}[(a)]
\item The simple modules $L_\la$ and $L_{\la -2\vare_k}$ lie in the same block, for~$1\le k\le n$. 
\item For~$\alpha \in \Phi_{\oo}^+$ with $s_\alpha\cdot\lambda<\lambda$, we have ${\mathrm{Hom}}_{\mc O}(M_{s_{\alpha}\cdot \la}, M_{\la}) \neq 0$. In particular, $L_{\la}$ and $L_{s_{\alpha}\cdot \la}$ lie in the same block.
\end{enumerate}

For (a) we observe that Lemma~\ref{LemChVInDV} implies that
$$[M_\lambda^\vee:L_{\lambda-2\varepsilon_k}]\;\ge\; [M_{\lambda-2\varepsilon_k}:L_{\lambda-2\varepsilon_k}]\;=\;1.$$
Equation~\eqref{topsoc} implies that~$M_\lambda^\vee$ is indecomposable with simple socle $L_\lambda$, which proves (a).

	For (b) we observe that \cite[Theorem~5.1(a)]{Hu08} implies a monomorphism $M_{s_\alpha\cdot\lambda}^{\oa}\hookrightarrow M_\lambda^{\oa}$. Applying the exact induction functor $\Ind_{\fg_{\ge 0}}^{\fg}$ yields the desired morphism.	
	\end{proof}


Now let $V:= \C^{n|n}$ be the natural representation. We have the corresponding exact endofunctor $-\otimes V$ of~$\cO$, which restricts to an endofunctor of~$\cO^\Delta$. Following~\cite{B+9}, we will use the ``fake Casimir element'' $\Omega$ to decompose the functor $-\otimes V$. This operator $\Omega$ also appeared in~\cite[Section~8.4]{Copn} and \cite[Section~2]{CP}. 
For the explicit realisation of~$\Omega \in \mf g \otimes \mf{gl}(n|n)$, we refer to \cite[Section~4.1]{B+9}. It decomposes as
$$\Omega = \Omega_{-1} + \Omega_0 + \Omega_1\quad\mbox{ with }\quad \Omega_i \in \mf g_i \otimes \mf{gl}(n|n)_{-i},\;\mbox{ for~$i=-1, 0, 1$.}$$
 For every~$\mf g$-module $M$, the $\Omega$-action on $M\otimes V$ commutes with~$\mf g$-action.  
 Consequently, for any $M\in\cO$, we have a decomposition
\begin{align} \label{DecomO}
&M\otimes V\;\cong\;\bigoplus_{z\in\C}(M\otimes V)_z,
\end{align}
where $(M\otimes V)_z$ is the generalised eigenspace for~$\Omega$, with eigenvalue $z$.

Define the shifted Weyl vector $\hat{\rho} : = \sum_{i=1}^n (1-i)\vare_i$. We also set
$$\check{\lambda}\;=\;\lambda+\hat{\rho}\;=\; \lambda+\sum_{i=1}^n (1-i)\vare_i,\quad\mbox{for all $\lambda\in\fh^\ast$.}$$
Note that  $\lambda\sim\mu$ if and only if $\check{\mu}$ can be obtained from $\check{\lambda}$ by repeatedly adding $a\varepsilon_i$ with~$a\in 2\mZ$ and $1\le i\le n$ and exchanging coefficients which have integer difference.

\begin{prop} \label{ppLemz}
	For $\la\in\fh^\ast$, we have 
	$$(M_\la\otimes V:M_\nu)\;=\;\begin{cases}
	1&\mbox{ if $\nu=\la\pm\varepsilon_j$, for some $1\le j\le n$,}\\
	0&\mbox{ otherwise.}
	\end{cases}$$
	Furthermore, for $z\in\mC$,
	\begin{enumerate}[(i)]
\item $M_{\lambda+\vare_j}$ appears as a subquotient in~$(M_{\lambda}\otimes V)_z$ if and only if $\check{\lambda}_j =z$; 
\item $M_{\lambda-\vare_j}$ appears as a subquotient in~$(M_{\lambda}\otimes V)_z$ if and only if $\check{\lambda}_j =z$. 
\end{enumerate}

\end{prop}
\begin{proof}
Denote the highest weight vector of~$M_\lambda$ by~$v$ and choose a basis $\{e_j, f_j\,|\, 1\le j\le n\}$ of~$V$ where $e_j$ has weight $\varepsilon_j$ and $f_j$ has weight $-\varepsilon_j.$  Since 
$$M_\lambda\otimes V\;\cong\; U\otimes_{U(\fb)}(\mC_\lambda\otimes V),$$
we have a filtration
$$0=N_{0}\subset N_{1}\subset\cdots\subset N_{2n-1}\subset N_{2n}=M_\lambda\otimes V,$$
where
$$N_i/N_{i-1}\;\cong\; M_{\lambda+\varepsilon_{i}}\quad\mbox{and}\quad N_{n+i}/N_{n+i-1}\;\cong\; M_{\lambda-\varepsilon_{1+n-i}},\quad\mbox{for~$1\le i\le n$.}$$
Furthermore, $v\otimes e_i+N_{i-1}$ generates $N_i/N_{i-1}$ and $v\otimes f_i+N_{n+i-1}$ generates $N_{n+i}/N_{n+i-1}$.

By \cite[Lemma~4.2.1(1)]{B+9}, we have $\Omega(v\otimes e_i)= \Omega_0(v\otimes e_i).$
By \cite[Lemma~4.2.1(2)]{B+9}, we thus find
$$\Omega(v\otimes e_i)\;\in\; (\lambda_i+1-i)(v\otimes e_i)+ N_{i-1}.$$
Similarly, with some additional straightforward computations, we have
$$\Omega(v\otimes f_i)\;\in \; \Omega_0(v\otimes f_i)-(n-1)(v\otimes f_i)+N_{n}\;\subset \; (\lambda_i+1-i)(v\otimes f_i)+ N_{n+i-1}.$$
Since $\Omega$ commutes with the action of~$\fg$, the claim about the generalised eigenvalues now follows.
\end{proof}

\begin{lem}\label{Lemlmequiv}
	Fix $\lambda,\mu\in\fh^\ast$ and $1\le i,j\le n$. If 
	$\lambda\sim\mu$ and $\check{\lambda}_i=\check{\mu}_j$,
	then we have $\lambda+\varepsilon_i\sim\mu+\varepsilon_j$.
\end{lem}
\begin{proof}
	The special case for~$i=j$ is obvious.
	
	Assume now that~$i\not=j$. The assumptions imply that $\lambda_i-\mu_j\in\mZ$ and $\mu_j-\lambda_j\in\mZ$. Consequently $s=s_{\varepsilon_i-\varepsilon_j}\in W^\lambda$. For $\nu:=s\cdot\lambda$, we thus find
	$$\nu\sim\lambda\sim\mu\quad\mbox{and}\quad \check{\nu}_j=\check{\lambda}_i=\check{\mu}_j.$$ 
	By the above special case we thus find $\nu+\varepsilon_j\sim \mu+\varepsilon_j$. Since $s\cdot(\nu+\varepsilon_j)=\lambda+\varepsilon_i$ the conclusion follows.
\end{proof}

\begin{cor}\label{Corz}
For each $\lambda\in \fh^\ast$ and $z\in\mC$, there exists $\nu\in\fh^\ast$, such that~$\theta:=(-\otimes V)_z$ restricts to a functor $\cO^\Delta([\lambda])\to\cO^\Delta([\nu])$.
\end{cor}
\begin{proof}
Assume first that~$[\lambda]$ contains no element $\kappa$ for which there is $1\le i\le n$ such that~$\check{\kappa}_i=z$. By Proposition~\ref{ppLemz}, $\theta$ is zero on $\cO^\Delta([\lambda])$, so there is nothing to prove.

Since we can replace $\lambda$ by any element in~$[\lambda]$, by the above we can assume that~$\check{\lambda}_i=z$, for some fixed  $1\le i\le n$.
By exactness of~$\theta$ it now suffices to prove that for each $\mu\in[\lambda]$, we have 
$\theta(M_\mu)\in \cO^\Delta([\lambda+\varepsilon_i]).$
This property follows from Proposition~\ref{ppLemz}, Lemma~\ref{Lemlmequiv} and the fact $\mu+\varepsilon_j\sim \mu-\varepsilon_j$.
\end{proof}

\begin{proof}[Proof of Theorem~\ref{ThmBlock}]
By Proposition~\ref{Prop::Linkage}, it suffices to prove that~$[P_\lambda:L_\mu]\not=0$ implies $\lambda\sim\mu$. We have
$$[P_\lambda:L_\mu]\;=\; \sum_{\zeta\in\fh^\ast}(P_\lambda:M_\zeta)[M_\zeta:L_\mu]\;\le\; \sum_{\zeta\in\fh^\ast}(P_\lambda:M_\zeta)(P_\mu :M_\zeta),$$
where the inequality follows from the combination of Lemma~\ref{LemChVInDV} and equation \eqref{BGG1}. Consequently, it actually suffices to prove that~$(P_\lambda:M_\mu)\not=0$ implies $\lambda\sim\mu$. By Corollary~\ref{ReDomi}, $P_\lambda$ is a direct summand of~$M_{\lambda'}\otimes V^{\otimes k}$, for some $k\in\N$ and $\lambda'\in\fh^\ast$. Since $P_\lambda$ is indecomposable, there must exist $\{z_l\in\C\}$ such that~$P_\lambda$ is a direct summand of
$$(\cdots ((M_{\lambda'}\otimes V)_{z_1}\otimes V)_{z_2}\otimes \cdots \otimes V)_{z_k}.$$
By Corollary~\ref{Corz}, we thus have $P_\lambda\in \cO^\Delta([\nu])$, for some $\nu\in\fh^\ast$. Since $(P_\lambda:M_\lambda)=1$, we have $\lambda\in[\nu]$ and the conclusion follows.
\end{proof}

The following lemma justifies our restriction to~$\cO_{\mZ}$ rather than $\cO_{\Upsilon}$ in Theorem~\ref{ThmBlock}(ii).
\begin{lem} \label{WtShf} For any $\la \in \h^*$ and $c\in \C$, we have an equivalence
	$$\mc O_{[\lambda]}\cong \mc O_{[\la +c\omega]}$$
	\end{lem} 
\begin{proof}
We have the Lie superalgebra morphism $\delta:\fg\to \mC$, with kernel $\mathfrak{spe}(n)$, defined by mapping each element to the trace of the matrix $a\in \mC^{n\times n}$, using the realisation in Subsection \ref{subsect::pen}. The morphism $c\delta$, for arbitrary~$c\in\mC$ thus yields a one-dimensional representation $\mC_c$ of~$\fg$ on which $\fh$ acts through $c\omega$.
 This yields an auto-equivalence $ -\otimes \C_{c} : \mc O\rightarrow \mc O$ 
	 with inverse $-\otimes \C_{-c}$, which restricts to the desired equivalence. 
\end{proof}

\subsection{Example: Generic blocks} Contrary to other types of Lie superalgebras, generic blocks in category $\cO$ for $\mathfrak{pe}(n)$ are not semisimple. In this section we pose some natural questions concerning their structure.

Let $t \in \C[\h^*]$ be the polynomial defined by  
$$ t(\mu): =\prod_{i<j}(\mu_i-\mu_j+j-i-1).$$  The following lemma generalises \cite[Lemma 3.2]{Se02} (also, see \cite[Corollary 5.8]{Se02}).
\begin{lem}  \label{IrrM}
	Let $\mu \in \h^*$. If $t(\mu) \neq 0$ we have $K_\mu=L_\mu$. If furthermore we have $M_\mu^\oa = L_\mu^\oa$ (i.e., $\mu$ is $\mf g_\oa$-antidominant), then 
	$$M_\mu=K_\mu=L_\mu\qquad\mbox{and}\qquad P_\mu=K^\vee_{\mu+2\omega}=M^\vee_{\mu+2\omega}.$$
\end{lem}
\begin{proof}
That $K_\mu=L_\mu$ when $t(\mu)\not=0$ follows from the exact same arguments as the proof of~\cite[Lemma~3.2]{Se02}. If $\mu$ is antidominant, we clearly have $M_\mu=K_\mu$ and $M^\vee_\mu=K^\vee_\mu$. Furthermore,
by equation~\eqref{topsoc2} it then follows that the top of~$K_{\mu+2\omega}^\vee=M^\vee_{\mu+2\omega}$ is $L_\mu=K_\mu$. Finally, ${\rm ch} P_{\mu}={\rm ch} M_{\mu+2\omega}^\vee$ by equation~\eqref{BGG2}. This implies that~$K_{\mu+2\omega}^\vee= P_{\mu}$.
\end{proof}

\begin{rem}
	If $\mu\in\Upsilon$ is antidominant, then $(-1)^{\frac{1}{2}n(n-1)}t(\mu)>0$, so the condition on $t(\mu)$ in Lemma~\ref{IrrM} becomes redundant.  
\end{rem}

\begin{Question} 
Consider a generic $\mu\in\fh^\ast$  in the sense that~$\mu_i-\mu_j\not\in\mZ$, for all $i\not=j$. Recall $S\subset\fh^\ast$ from Lemma~\ref{LemChVInDV}. Then Lemma \ref{IrrM} and equation~\eqref{BGG3} imply
	 $$[P_\mu:L_\nu]=\begin{cases} 1 &\mbox{ if $\nu-\mu\in S$,}\\
	 0&\mbox{ otherwise.}
	 \end{cases}$$
\begin{enumerate}[(i)]
\item  Is the radical filtration of~$P_\mu$ given by~$$\text{rad}^{\ell}P_\mu/ \text{rad}^{\ell+1}P_\mu \cong \bigoplus_{1\leq i_1<i_2<\cdots <i_{\ell}\leq n} L_{\mu+2\sum_{k=1}^{\ell}\vare_{i_k}}, \text{ for } \ell \in \N\cup \{0\}?$$ \\   
\item We have the set $I:=\{1,2,\cdots,n\}$.
Let $\Z^{\oplus I}\cong \mZ^{\oplus n}$ be the free abelian group with basis $\{e_i\,|\, e_i\in I\}$. We define the quiver $Q$ with vertices $Q_0=\mZ^{\oplus I}$ and edges given by 
$$Q_1=\{x^{(i)}_v:v\to v+e_i\,|\, v\in \mZ^{\oplus I}\mbox{ and }i\in I\}.$$
	Let $A^n$ be the path algebra of~$Q$ with relations 	
$$x^{(i)}_{v+e_i} x^{(i)}_{v}=0\quad\mbox{and}\quad x^{(i)}_{v+e_j}x^{(j)}_v=x^{(j)}_{v+e_i}x^{(i)}_v,\quad\mbox{for all $v\in \mZ^{\oplus I}$ and $i,j\in I$.}$$
Do we have an algebra isomorphism
$$A^n\;\stackrel{\sim}{\to}\; \text{End}^{\text{fin}}_{\mc O}(\bigoplus_{\lambda \in [\mu]} P_\lambda)\qquad\mbox{with }\quad 1_v\mapsto \id_{P_{\mu+2v}},$$
with $v=\sum_iv_ie_i\in \mZ^{\oplus I}$ interpreted as $\sum_iv_i\varepsilon_i\in\fh^\ast$?
	
		Note that we have $A^n\cong A^{\otimes n}$, with $A=A^1$ the path algebra of the quiver with edges labelled by $\mZ$
	$$\xymatrix{
	\cdots\ar[r]^{x_{-2}}&\bullet_{-1}\ar[r]^{x_{-1}}&\bullet_0\ar[r]^{x_0}&\bullet_1\ar[r]^{x_1}&\cdots
	}$$
	and relations $x_ix_{i-1}=0$ for $i\in\mZ$. In particular, $A^n$ is Koszul, with grading given by putting the arrows of the quiver in degree $1$.
	\item Let $\nu \in \mf h^*$ such that  $\nu_i -\nu_j \notin \Z$ for all $1\leq i\neq j\leq n$. Is it true that~$\mc O_{[\mu]} \cong \mc O_{[\nu]}$?
	\end{enumerate}
\end{Question}

We will answer these questions in the affirmative for $\mathfrak{pe}(2)$ in the following section.

\section{Category~$\cO$ and primitive ideals for~$\mathfrak{pe}(2)$}\label{Sec2}

\subsection{Characters of simple modules}
 
\begin{lem} \label{Lem::ComVer} 
	Let $s = s_{\vare_1-\vare_2}$ be the simple reflection associated to the simple root $\vare_1 -\vare_2$. We have the following composition factors of Verma modules.
	\begin{enumerate}[(i)]
	\item If $\mu_1 - \mu_2\in\mC\backslash \Z_{\ge 0}$ then $M_\mu =L_\mu$.
	\item If $\mu_1-\mu_2 \in \Z_{> 0}$ then $\emph{ch}M_\mu =\emph{ch}L_\mu +\emph{ch}L_{s\cdot\mu}.$
	\item If $\mu_1 = \mu_2$ then $\emph{ch}M_\mu 
		=\emph{ch}L_\mu + \emph{ch}L_{s\cdot \mu} + \emph{ch}L_{\mu-\omega}$. 
		\end{enumerate}
		\end{lem}
\begin{proof} We first note that part (i) follows from Lemma \ref{IrrM}. 
	
	
	Now we suppose that~$\mu_1-\mu_2 \in \Z_{> 0}$. Then $K_\mu = L_\mu$ and $K_{\mu -\omega}  =L_{\mu -\omega}$ by  Lemma \ref{IrrM}, which means that 
	$$\text{ch}L_\mu = \text{ch}L_\mu^\oo +\text{ch}L_{\mu-\omega}^\oo,$$ since $e^{-\omega}\text{ch}L_\zeta^\oo =\text{ch}L_{\zeta-\omega}^\oo $ for all $\zeta \in \h^*$. Also, we note that \begin{align*}
	&\text{ch}M_\mu = \text{ch}M^\oa_\mu+e^{-\omega}\text{ch}M^\oa_\mu = \text{ch}L^\oo_\mu+ \text{ch}L^\oo_{s\cdot \mu} + \text{ch}L^\oo_{s\cdot (\mu-\omega)} + \text{ch}L^\oo_{\mu-\omega}.
	\end{align*}
	Part (ii) thus follows.
	
	Finally, we assume that~$\mu_1 =\mu_2$, which means that~$\mu=a\omega$, for some $a\in\mC$. We may observe that~$L_{a\omega}$ is one-dimensional for each $a\in \C$. We may conclude that have 
	\begin{align*}
	&\text{ch}M_\mu = \text{ch}L^\oo_\mu+ \text{ch}L^\oo_{s\cdot \mu} + \text{ch}L^\oo_{s\cdot (\mu-\omega)} + \text{ch}L^\oo_{\mu-\omega}
	=\text{ch}L_\mu + \text{ch}L_{s\cdot \mu} + \text{ch}L_{\mu-\omega},
	\end{align*} where we used part(i) to calculate $\text{ch}L_{s\cdot \mu}$. 
\end{proof}

We give the irreducible characters for~$\mf{pe}(2)$ as follows:
\begin{cor} \label{pe2char} Let $\mu \in \h^*$.  
\begin{enumerate}[(i)]
	\item If $\mu_1 -\mu_2 \notin \Z_{\geq 0}$ then $L_\mu =M_\mu$.
	\item If $\mu_1 -\mu_2 \in \Z_{> 0}$ then $\emph{ch}L_\mu =\emph{ch}M_\mu - \emph{ch}M_{s\cdot \mu}$. 
    \item If $\mu_1 =\mu_2 $ then $\emph{ch}L_\mu = e^{\mu}$. 
\end{enumerate}\end{cor}

\subsection{Characters of projective modules}
 \label{SetchP}

 Let $\ell(M)$ denote the length of a composition series of a module $M$.
 \begin{cor} \label{ChProj}
 \begin{enumerate}[(i)]
 \item If $\la_1 -\la_2 \notin \Z$ then we have $\ell(P_\la) = 4$.
 
\item	If $\la_1-\la_2 \in \Z$, we have
	$$\ell(P_\la)= \left\{ \begin{array}{llll} 12, \text{ if $\la_1 +4 < \la_2$.} \\
	13, \text{ if $\la_1 +4=\la_2$}.\\
	11,\text{ if $\la_1+3= \la_2$}. \\
	15,\text{ if $\la_1+2= \la_2.$}\\
	5,\text{ if $\la_1+1= \la_2.$}\\
	18,\text{ if $\la_1= \la_2.$}\\
	7,\text{ if $\la_1-1= \la_2.$}\\
	9,\text{ if $\la_1-2= \la_2.$}\\
	8,\text{ if $\la_1-2> \la_2.$}\\
	\end{array} \right.
	$$

	\end{enumerate}
\end{cor}
\begin{proof} We let $[M]$ denote the image of a module $M$ in the Grothendieck group. First assume that~$\la_1-\la_2 \notin \Z$. Equation
\begin{equation}\label{P2gen}[P_\la]  = [M^\vee_{\la+2\omega} ]=  [L_{\la}] + [L_{\la + 2 \vare_1}] +   [L_{\la + 2 \vare_2}]+  [L_{\la + 2\omega}]\end{equation} follows from \eqref{BGG2}, \eqref{BGG3} and Lemma~\ref{Lem::ComVer}(i) and implies part (i). 
	
	Now we assume that~$\la_1-\la_2 \in \Z$. 	The lengths of projective covers in~$\mc O$ over $\mf{pe}(2)$ follow from direct computations by Proposition \ref{Thm::BGGrcp} and Lemma \ref{Lem::ComVer}.

 If $\la_1+4 <\la_2$ then 
\begin{align*}
&[P_\la] = [M_{\la+2\omega}^{\vee}] + [M_{s\cdot \la+2\omega}^{\vee}] \\ &\hskip 0.73cm=2 [L_{\la+ 2\omega}] + 2[L_{\la + 2 \vare_2}] + 2  [L_{\la + 2 \vare_1}]+ 2 [L_{\la }] +[L_{s\cdot \la+ 2\omega}] + [L_{s\cdot \la + 2 \vare_2}]+ [L_{s\cdot \la +2\vare_1}]+[L_{s\cdot \la }].
\end{align*} 


 If $\la_1+4 = \la_2$ then 
 \begin{align*}
 &[P_\la] =[M_{\la+2\omega}^{\vee}] + [M_{s\cdot \la+2\omega}^{\vee}] \\ &\hskip 0.73cm=2 [L_{\la+ 2\omega}] + 2[L_{\la + 2 \vare_2}] +  2 [L_{\la + 2 \vare_1}]+ 2 [L_{\la }]  \\ &\hskip 0.73cm +[L_{s\cdot \la+ 2\omega}] + [L_{s\cdot \la + 2 \vare_2}] +  [L_{s\cdot \la + 2 \vare_1}]+  [L_{s\cdot \la }]+[L_{s\cdot \la -\vare_1+\vare_2}].
 \end{align*}
 
 If $\la_1+3 =\la_2$ then 
 \begin{align*}
 &[P_\la] =[M_{\la+2\omega}^{\vee}] + [M_{s\cdot \la+2\omega}^{\vee}] \\ &\hskip 0.73cm=2 [L_{\la+ 2\omega}] +2 [L_{\la + 2 \vare_2}] +  [L_{\la + 2 \vare_1}]+ 2 [L_{\la }] + [L_{s\cdot \la+ 2\omega}] + [L_{s\cdot \la + 2 \vare_2}] + [L_{s\cdot \la + 2 \vare_1}] +   [L_{s\cdot \la }].
 \end{align*}
 
 If $\la_1+2 =\la_2$ then 
 \begin{align*}
 &[P_\la] = [M_{\la+2\omega}^{\vee}] + [M_{s\cdot \la+2\omega}^{\vee}] \\ &\hskip 0.73cm= 2[L_{\la+ 2\omega}] + 2[L_{\la + 2 \vare_2}] +   [L_{\la + 2 \vare_1}] + 2[L_{s\cdot \la + 2 \vare_2}] \\ &\hskip 0.73cm+ [L_{\la + \vare_1 -\vare_2}] + 2 [L_{\la }] +[L_{s\cdot \la+ 2\omega}] +[L_{s\cdot \la+\omega}] + [L_{s\cdot \la + 2 \vare_1}] + [L_{s\cdot \la }] + [L_{s\cdot \la -\omega }].
 \end{align*}
 
 If $\la_1+1 = \la_2$ then 
 \begin{align*}
 &[P_\la] = [M_{\la+2\omega}^{\vee}]  = [L_{\la+2\omega}] + [L_{\la+2\vare_2}]+ [L_{\la+2\vare_1}] + [L_{s\cdot \la+2\vare_2}]+[L_{\la}].
 \end{align*}
 
 If $\la_1 = \la_2$ then \begin{align*}
 &[P_\la] = [M_{\la+2\omega}^{\vee}] + [M_{\la+3\omega}^{\vee}] \\
 &\hskip 0.73cm=[L_{\la+2\omega}]+[L_{s\cdot \la +2\omega}]+[L_{\la+\omega}] + [L_{\la+2\vare_1}]+[L_{s\cdot \la +2\vare_2}]    \\
 &\hskip 0.73cm+ [L_{\la+2\vare_2}] +[L_{\la}]+ [L_{s\cdot \la}] +[L_{\la-\omega}]+[L_{\la+3\omega}]+[L_{s\cdot \la +3\omega}]+[L_{\la+2\omega}]  \\ & \hskip 0.73cm+[L_{\la+3\vare_1+\vare_2}]+[L_{s\cdot \la +\vare_1+3\vare_2}] + [L_{\la+\vare_1+3\vare_2}] +[L_{\la+\omega}] + [L_{s\cdot \la+\omega}] +[L_{\la}].
 \end{align*}

 If $\la_1 -1 = \la_2 $ then 
 \begin{align*}
 &[P_\la] = [M_{\la+2\omega}^{\vee}] = [L_{\la+2\omega}] + [L_{s\cdot \la+2\omega}] + [L_{\la +2\vare_2}]+[L_{\la+2\vare_1}] + [L_{s\cdot  \la +2\vare_2}] + [L_{\la}] + [L_{s\cdot \la}].
 \end{align*}

 If $\la_1 -2 = \la_2 $ then
 \begin{align*}
 &[P_\la] = [M_{\la+2\omega}^{\vee}] \\ &\hskip 0.73cm= [L_{\la+2\omega}] + [L_{s\cdot \la+2\omega}] + [L_{\la +2\vare_2}] +[L_{s\cdot \la +2\vare_1}]+ [L_{\la-\alpha}]+[L_{s\cdot \la}] +[L_{\la+2\vare_1}] + [L_{s\cdot  \la +2\vare_2}] + [L_{\la}].
 \end{align*}
 
 If $\la_1 -2 > \la_2 $ then
 \begin{align*}
 &[P_\la] = [M_{\la+2\omega}^{\vee}] \\ &\hskip 0.73cm= [L_{\la+2\omega}] + [L_{s\cdot \la+2\omega}] + [L_{\la +2\vare_2}] +[L_{s\cdot \la +2\vare_1}]+[L_{\la+2\vare_1}] + [L_{s\cdot  \la +2\vare_2}] + [L_{\la}] + [L_{s\cdot \la}]. 
 \end{align*}
 
 This concludes the proof.
\end{proof}

\begin{rem}
It is proved in~\cite[Theorem 7.1.1]{B+9} that projective covers in~$\mc F$ are sent to projective covers or zero by the translation functor defined in Corollary \ref{Corz}. However, in $\cO$, already for~$\mf{pe}(2)$ there are translated projective covers which are decomposable. For example, Lemma \ref{LemChVInDV} and Proposition~\ref{ppLemz} allow to show that
$$(P_0\otimes V)_{z=2}\;\cong\; P_{\vare_1}\oplus P_{\vare_1+2\vare_2}.$$
Another observation is that in $\cO$ we no longer have $[P_\lambda:L_\mu]\le 1$, contrary to \cite[Theorem~8.1.2]{B+9}.
\end{rem}

\subsection{Equivalence of blocks}
\begin{thm} The BGG category~$\mc O$ for $\mf{pe}(2)$ has exactly $3$ blocks up to equivalence. Concretely:
\begin{enumerate}[(i)]
\item $\mc O_{[\partial^0]} \not \cong \mc O_{[\partial^1]} \cong \mc O_{[\partial^2]}$.
\item Let $\la, \mu \in \h^*$ with~$\la_1 -\la_2, ~\mu_1-\mu_2 \notin \Z$, then 
	$$\mc O_{[\la]}\cong \mc O_{[\mu]}\quad\mbox{and}\quad \mc O_{[\la]}\not \cong \mc O_{[\partial^i]}, \text{ for }i=0,1, 2.$$
	\end{enumerate}
	Furthermore, $\cO_{[\lambda]}$ is Koszul, whenever $\la_1 -\la_2\notin \Z$.
\end{thm}
\begin{proof}
Parts (i) and (ii), together with Lemma \ref{WtShf} imply that there are 3 blocks up to equivalence.
	Lemma \ref{WtShf} also implies the equivalence in part (i) since
	$$\mc O_{[\partial^1]}\cong \mc O_{[\partial^1+\omega]} =\mc O_{[\partial^2]}.$$
   By Theorem~\ref{ThmBlock} and Corollary  \ref{ChProj}, we have  
	\begin{align*}
	&\text{max}\{\ell(P_\zeta)|~L_\zeta \in \mc O_{[\la]} \} =4,\\
	&\text{max}\{\ell(P_\zeta)|~L_\zeta \in \mc O_{[\partial^0]}\} =18, \\
	&\text{max}\{\ell(P_\zeta)|~L_\zeta \in \mc O_{[\partial^1]}\}   = 12,
	\end{align*}
	which proves all the non-equivalences.
	
	It remains to show $\mc O_{[\la]}\cong \mc O_{[\mu]}$ for $\lambda$ and $\mu$ as in part (ii). To see this, we first recall from Corollary \ref{ChProj} that~$$\text{ch}(\text{rad}P_\la) =   \text{ch}L_{\la + 2 \vare_1} +   \text{ch}L_{\la + 2 \vare_2}+  \text{ch}L_{\la + 2\omega}.$$
	We claim the radical filtration of~$P_\lambda$ satisfies
	$$\text{rad}P_\la/\text{rad}^2P_\la \cong L_{\la +2\vare_1} \oplus L_{\la +2\vare_2}\quad\mbox{and}\quad \text{rad}^2P_\la = \text{Soc}P_\la \cong L_{\la +2\omega}.$$
	    By Lemma~\ref{IrrM} and equation~\eqref{topsoc}, we have $\Soc P_\la=L_{\la+2\omega}$.
    As a consequence, it suffices to show that
    $$\text{Ext}^1_{\mc O}(L_{\la +2\vare_1}, L_{\la +2\vare_2})  = \text{Ext}^1_{\mc O}(L_{\la +2\vare_2}, L_{\la +2\vare_1})=0.$$
    Indeed, if $\text{Ext}^1_{\mc O}(L_{\la +2\vare_1}, L_{\la +2\vare_2})\neq 0$, then $[P_{\la +2\vare_1}:L_{\la +2\vare_2}]\neq 0$, which is contradicted by~\eqref{P2gen}. Similarly, we have $\text{Ext}^1_{\mc O}(L_{\la +2\vare_2}, L_{\la +2\vare_1})=0$. 
    
     It suffices to show that we have an equivalence between the respective categories of projective modules. Equivalently, it suffices to show that there is an isomorphism between locally finite endomorphism algebras $\text{End}^{\text{fin}}_{\mc O}(\bigoplus_{\gamma \sim \la} P_{\gamma}) $ and $ \text{End}^{\text{fin}}_{\mc O}(\bigoplus_{\gamma \sim \mu} P_{\gamma})$.      
     By a direct computation these algebras are isomorphic to the path algebra of the following quiver with vertices $\mZ\times\mZ$:
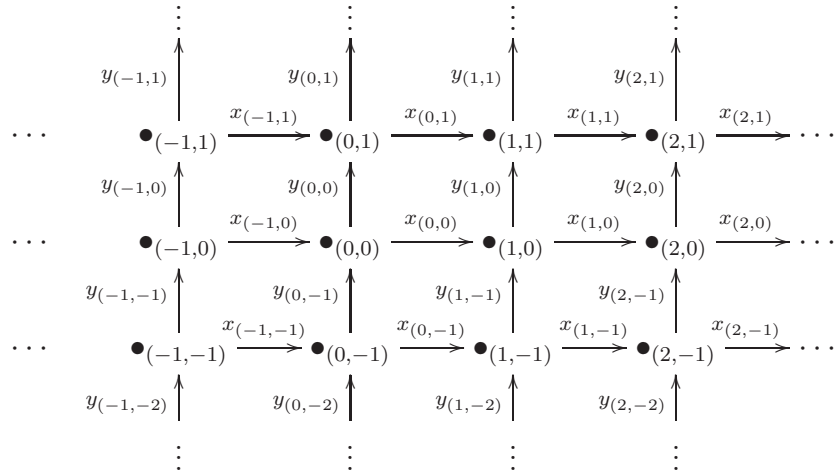
     \begin{displaymath}
    \xymatrix{
    &\vdots&\vdots&\vdots&\vdots&\\
       \cdots  &\bullet_{(-1,1)}\ar[r]^{x_{(-1,1)}}\ar[u]^{y_{(-1,1)}}&\bullet_{(0,1)}\ar[r]^{x_{(0,1)}}\ar[u]^{y_{(0,1)}}& \bullet_{(1,1)} \ar[r]^{x_{(1,1)}}\ar[u]^{y_{(1,1)}}&  \bullet_{(2,1)}\ar[r]^{x_{(2,1)}}\ar[u]^{y_{(2,1)}} & \cdots  \\
        \cdots &\bullet_{(-1,0)}\ar[r]^{x_{(-1,0)}}\ar[u]^{y_{(-1,0)}}&\bullet_{(0,0)}\ar[r]^{x_{(0,0)}}\ar[u]^{y_{(0,0)}}& \bullet_{(1,0)} \ar[r]^{x_{(1,0)}}\ar[u]^{y_{(1,0)}}&  \bullet_{(2,0)}\ar[r]^{x_{(2,0)}}\ar[u]^{y_{(2,0)}} & \cdots  \\
           \cdots &\bullet_{(-1,-1)}\ar[r]^{x_{(-1,-1)}}\ar[u]^{y_{(-1,-1)}}&\bullet_{(0,-1)}\ar[r]^{x_{(0,-1)}}\ar[u]^{y_{(0,-1)}}& \bullet_{(1,-1)} \ar[r]^{x_{(1,-1)}}\ar[u]^{y_{(1,-1)}}&  \bullet_{(2,-1)}\ar[r]^{x_{(2,-1)}}\ar[u]^{y_{(2,-1)}} & \cdots  \\
           &\vdots\ar[u]^{y_{(-1,-2)}}&\vdots\ar[u]^{y_{(0,-2)}}& \vdots\ar[u]^{y_{(1,-2)}}& \vdots\ar[u]^{y_{(2,-2)}}
        }
\end{displaymath}
and relations 
$$y_{v+(1,0)}x_{v} =x_{v+(0,1)}y_v\quad\mbox{and}\quad
    x_{v+(1,0)}x_v=0 =y_{v+(0,1)}y_v =0,\qquad\mbox{for all $v\in \mZ\times\mZ$}.$$
  In particular these algebras do not depend on the specific $\lambda\in\fh^\ast\backslash\Upsilon$ and are Koszul.  
\end{proof}

\begin{rem} 
	In~\cite{CMW}, Cheng, Mazorchuk and Wang proved that non-integral blocks of category~$\cO$ for $\mathfrak{gl}(m|n)$ are equivalent to integral blocks in~$\cO$ for direct sums of smaller general linear superalgebras. The results in this 
	section similarly imply that non-integral~$\mf{pe}(2)$-blocks are equivalent to integral blocks (in~$\cO$ or equivalently $\cF$) for $\mf{pe}(1)\oplus \mf{pe}(1)$.
However, there is no realisation of~$\mf{pe}(1)\oplus \mf{pe}(1)$ as a subalgebra $\fk\subset\mathfrak{pe}(2)$ for which we have a Borel subalgebra $\fb\subset\mathfrak{pe}(2)$ such that $\fk+\fb$ constitutes a subalgebra. It is thus not possible to consider parabolic induction as in~\cite{CMW} to prove the equivalence directly.	\end{rem}

\subsection{The primitive spectrum}

\begin{lem} \label{pe2TwIrr} Let $\mf g := \mf{pe}(2)$ and $T:=T_s$ be the twisting functor with the (unique) simple reflection $s$ of~$W$.
	Then we have the following character formulas ($\omega = \vare_1 + \vare_2$):
	\begin{align}
	&T L_\la = L_{s\cdot \la}, ~\text{ if $\la_1-\la_2 \notin \Z$}, \label{pe2Tw1} \\
	&T L_\la = 0, ~\text{ if $\la_1-\la_2 \in \Z_{\geq 0}$},\label{pe2Tw2} \\
	&T L_\la = L_\la, ~\text{ if $\la_2 = \la_1+1$},\label{pe2Tw3} \\
	&\emph{ch}T L_{\la} = \emph{ch}L_{\la}+\emph{ch}L_{s\cdot \la}+ \emph{ch}L_{s\cdot \la -\omega}, ~\text{ if $\la_2= \la_1+2$}, \label{pe2Tw4}\\
	&\emph{ch}T L_{\la} = \emph{ch}L_{\la}+\emph{ch}L_{s\cdot \la}, ~\text{ if $\la_2> \la_1+2$}. \label{pe2Tw5}
	\end{align}
\end{lem}
\begin{proof}
	Recall that~$\Res M_\mu = M_\mu^\oo \oplus M_{\mu- \omega}^\oo$, for all $\mu \in \h^*$. Then it follows from $T^\oo \circ \Res = \Res \circ T$ and $\text{ch}T^\oo M_\mu^\oo = \text{ch}M^\oo_{s\cdot \mu}$ for all $\mu \in \h^*$ that~$$\text{ch}T M_\la =\text{ch}T^\oo \Res  M_\la  = \text{ch}M_{s\cdot \la} = \text{ch}M_{(\la_2-1)\vare_1 + (\la_1 +1)\vare_2}.$$

	We first consider \eqref{pe2Tw1}, that is, assume that~$\la_1 - \la_2 \notin \Z$. In this case we have $M_\la = L_\la$ and $M_{s\cdot \la} = L_{s\cdot \la} $ as desired. Also, \eqref{pe2Tw2} follows from \cite[Theorem 5.12(i)]{CoMa14} since  $L_\la$ is finite-dimensional. 
	
	Now consider \eqref{pe2Tw3}, that is, assume that~$\la_2 = \la_1+1$. Observe that~$s\cdot \la  =\la$ and so $$\text{ch}T M_\la = \text{ch}M_{s \cdot \la} = \text{ch}M_{\la_1\vare_1+ (\la_1+1)\vare_2} = \text{ch}L_{\la}.$$
	
	Then consider \eqref{pe2Tw4}. That is, $\la_2= \la_1+2$ and note that~$s\cdot\la =  (\la_2 -1)\vare_1+ (\la_1+1)\vare_2$ with~$(\la_2 -1)= (\la_1+1)$. Then $\text{ch}M_{s\cdot \la} = \text{ch}L_\la + \text{ch}L_{s\cdot \la} + \text{ch}L_{s\cdot \la -\omega}$. We have thus proved \eqref{pe2Tw4}. 
	
	Finally, we consider \eqref{pe2Tw5}, that is, assume that~$\la_2> \la_1+2$. In this case we have $M_\la =L_\la$. 
	Therefore we have $\text{ch}TL_\la = \text{ch}M_{s\cdot \la} = \text{ch}L_{s\cdot \la} + \text{ch}L_{\la}$. 
\end{proof}

Let $\mu\in \h^*$. It follows from equation \eqref{dualchar} that $\text{ch}D_\sigma M_\mu =\text{ch}M_{-s\mu+\omega}$. Therefore by Corollary \ref{pe2char} we have the following formulas,
$$D_\sigma L_\mu  \;=\;\begin{cases}
L_{-s\mu}&\mbox{ if $\mu_1=\mu_2$,}\\
L_{-s\mu+\omega}&\mbox{ otherwise.}
\end{cases}$$

By Lemma \ref{pe2TwIrr} and Corollary~\ref{TwPI}, we have the following description of the primitive spectrum of~$\mf{pe}(2)$.

\begin{cor} For $a,b\in\mC$, we set $J(a,b):=\Ann_UL(a\varepsilon_1+b\varepsilon_2)$. We have the following connected components of the inclusion order on $\{J(a,b)\,|\, (a,b)\in\mC\times\mC\}$:
\begin{enumerate}[(i)]
\item For each $\{a,b\}\subset\mC$ with $a-b\not\in\mZ$, we have $J(a,b)=J(b-1,a+1)$.
\item For each $a\in\mC$, we have the singleton $\{J(a,a+1)\}$.
\item For each $a\in\mC$, the set $\{J(a+i,a+i), J(a+i,a+i+2)\,|\, i\in\mZ\}$ is a connected component with Hasse diagram
$$
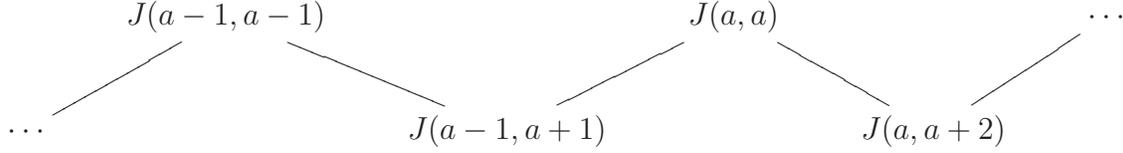
\xymatrix{
	&J(a-1, a-1)\ar@{-}[ld] \ar@{-}[rd]  & &J (a,a) \ar@{-}[rd] && \cdots\\
	\cdots&&J (a-1, a+1) \ar@{-}[ru] && J(a,a+2) \ar@{-}[ru] & }$$
	\item For each $a\in\mC$ and $k\in \mZ_{>0}$, we have
	$$J(a-1,a+k+1)\subsetneq J(a+k,a).$$
\end{enumerate}
\end{cor}



\end{document}